\documentclass{article}
\usepackage[english]{babel}
\usepackage[letterpaper,top=2cm,bottom=2cm,left=3cm,right=3cm,marginparwidth=1.75cm]{geometry}

\usepackage{amsmath}
\usepackage{amssymb}
\usepackage{graphicx}
\usepackage{bm}
\usepackage[dvipsnames]{xcolor} 
\usepackage{mathtools}
\usepackage{amsthm} 
\newtheorem{thm}{Theorem}
\newtheorem{lem}{Lemma}
\newtheorem{prop}{Proposition}

\newtheorem{rem}{Remark}
\newtheorem{ass}{Assumption}

\newcommand{\Rz}{\mathbb{R}}
\newcommand{\Nz}{\mathbb{N}}
\newcommand{\Xvec}{{\mathbf X}_{\bullet,\bullet}}
\newcommand{\Yvec}{{\mathbf Y}_{\bullet,\bullet}}
\newcommand{\uvec}{{\mathbf u}_{\bullet,\bullet}}
\newcommand{\Rbig}{{\mathbf R}^{-1}}
\newcommand{\wPhi}{\widehat\Phi_{m}}
\newcommand{\wPsi}{\widehat\Psi_m^{(R)}}
\newcommand{\ePsi}{\Psi_m^{(R)}}
\newcommand{\E}{\mathbb E}

\usepackage[colorlinks=true, allcolors=OliveGreen]{hyperref}
\usepackage{cleveref}

\usepackage{natbib}
\renewcommand{\cite}{\citet}

\title{Adaptive nonparametric regression from repeated measurements under common noise}
\author{{\sc Fabienne Comte}\thanks{Universit\'e Paris Cit\'e, MAP5 UMR 8145, F-75006 Paris, France, e-mail:
  \url{fabienne.comte@u-paris.fr}} \ , { \sc Bianca Neubert}\thanks{Ruprecht-Karls-Universität Heidelberg, Institut f\"ur Mathematik, M$\Lambda$THEM$\Lambda$TIKON, Im Neuenheimer Feld 205,
  D-69120 Heidelberg, Germany, e-mail:
  \url{neubert@math.uni-heidelberg.de}}}
\date{}

\begin{document}
\maketitle
\begin{abstract}
We consider nonparametric estimation of the regression function in a model where individuals share a common noise component and repeated measurements are available for each individual. We propose a projection estimator which minimizes a least-squares contrast that accounts for the covariance structure resulting from the common noise. We analyze its risk measured either as the expectation of the empirical norm or as the expectation of the theoretical norm associated with the contrast. We discuss how the number of repeated measurements affects the estimation rates in the common noise model, and precisely characterize the dependence on the number of repetitions. In addition, we propose a data-driven projection estimator and establish risk bounds in terms of the expected empirical norm. The results are illustrated with some simulation experiments.
\end{abstract}
{\footnotesize
\begin{tabbing} 
\hspace{4.7em}  Repeated measurements\\[.2ex] 
\end{tabbing}}

\section{Introduction}

In this paper, we consider the following model
\begin{equation}\label{model} Y_{i,j} = b(X_{i,j}) + \sqrt{\rho}\varepsilon_{0,j} + \sqrt{1-\rho}\varepsilon_{i,j}
\end{equation}
for $i=1, \dots, n$ and $j=1, \dots, J$, which means that for each individual $i$ we have $J$ repeated observations of the model. The $X_{i,j}$ are independent and identically  distributed (i.i.d.) real-valued random variables with density $f$, the $\varepsilon_{i,j}$ are i.i.d. and centered, with variance $\sigma^2>0$ and independent of the $(X_{i,j}: i=1,\ldots n; j=1,\ldots,J)$. The parameter $\rho\in (0,1)$ describes the correlation structure of the noise and is assumed to be known.  The regression function $b\colon\Rz \rightarrow\Rz$ is unknown and needs to be estimated. We assume that $b$ is square integrable and prove an upper bound on the risk of an adaptive nonparametric projection estimator. 

The model is simple and interesting in itself, and refers to a context where a structural noise $\varepsilon_{0,j}$ is common to all individuals $i$, in addition to the standard one $\varepsilon_{i,j}$, and repeated measurements $j=1, \dots, J$ are available for each individual. However, the truth is that we were first interested in model (\ref{model}) as a simplified version of a Stochastic Differential Equation (SDE) model intensively studied in recent literature of probability and statistical inference for stochastic processes.   The model arose as a way to capture the effect of some random external force influencing simultaneously all individuals $i$ such as environment or constraints of energetic or economic type. This idea has been developed in finance (see {\it e.g.} \cite{duellmann2010estimating}) or  in the theory of mean field games and control, for models of SDEs with common noise (see {\it e.g.} \cite{carmona2016mean}, \cite{lacker2022superposition}, \cite{maillet2025notelongtimebehaviourstochastic}, \cite{delarue2024ergodicity}). Indeed, these authors consider a system of SDEs of the form:
\begin{equation*} 
dX_i(t)= b(X_i(t)) \,dt + \sigma(\sqrt{\rho}dW_0(t) +\sqrt{1-\rho}dW_i(t)),\quad X_i(0)=x, \quad i=1, \ldots,n, 
\end{equation*}
where $W_0, W_1, \ldots, W_n$ are $n+1$ independent Brownian motions. Thus, $W_0$ is common to all equations and this creates identically distributed but correlated processes.

Now, imagine that you have discrete time observations of the $n$ paths, $X_i(t_j)$, for $t_j=j/J$, and $i=1, \dots, n$. Then setting $Y_{i,j}=J(X_i(t_{j+1})-X_i(t_j))$, for $J$ large enough, the approximation 
$$Y_{i,j}\simeq b(X_i(t_j)) + \sqrt{\rho}\sigma J (W_0(t_{j+1})-W_0(t_j)) + \sqrt{1-\rho} \sigma J(W_i(t_{j+1})-W_i(t_j)),$$
gives a regression model. This leads to the simplified model described in (\ref{model}).

As already mentioned, Model (\ref{model}) is worth being studied as an interesting extension of the standard regression problem. That is, the estimation of the regression function in the nonparametric setting has a long history, and we refer to \cite{TsyBook} or  \cite{EfroBook} for general presentations of the topic and the associated estimation strategies. However, the literature mainly deals with the case $\rho=0$.  The present model involves a specificity in the strong dependency involved by the common noise feature corresponding to $\rho\neq 0$, which makes results related to classical dependent settings inapplicable, see  \cite{CaronDedMic} and the references therein. Indeed, in \cite{CaronDedMic}, the spectral radius of the noise variance matrix must be uniformly bounded, which is not the case in our model (it has  order $O(n)$, see our matrix $R$ below in Section \ref{sec::general-notation}); or the correlation between error terms must tend to 0, which is not the case here neither; or the variance of the sum of noises must be at most of order $nJ$, while here, it is equal to $\sigma^2 (nJ)(n \rho + 1-\rho)$, which still tends to infinity when divided by $nJ$. Regression models with dependent noise have also been considered in \cite{CM95}, \cite{BCV2001} or \cite{FP2007} but the present fixed dependency violates all standard assumptions and in particular the ones in these papers (decreasing correlations in \cite{CM95}, decreasing mixing coefficients in \cite{BCV2001} or uncorrelated noise in \cite{FP2007}). 

A previous paper considers a similar context and introduces tools which are reused here:  \cite{comte2025non} study model (\ref{model}) with $J=1$, and build a nonadaptive projection estimator. In the present work, we generalize their results to $J\geq 1$. In some sense, we go from $n$ to $nJ$ observations and everything may at first sight be changed in accordance. But this is not exactly the case. Indeed, the risk bound is modified from a decomposition into the squared bias and a variance term of order dim$(S_m)/n$ to the squared bias and a variance term of order dim$(S_m)/(nJ)$, where $S_m$ denotes the finite dimensional projection space on which the estimator is built. But the stability condition associated to the matrix to be inverted for evaluating the least-squares estimator (see \cite{CohenDavLev}) only changes from $n$ to $n\sqrt{J}$ and not $nJ$, roughly speaking. The fact that different rates arise is somehow in accordance with parametric results obtained in  \cite{GL2026}; therein, if $b=b_\theta$ and the noise is re-parameterized by $\sigma( \sqrt{\rho}dW_0(t)+\sqrt{1-\rho}dW_i(t))=\sqrt{v}(dW_i(t)+ \sqrt{c} dW_0(t))$,  then the parameters $(v,c,\theta) $ are estimated at rates $(\sqrt{nJ}, \sqrt{J}, \sqrt{n})$, the variance of Brownian increments $W_i(t_{j+1})-W_i(t_j)$ being equal to $1/J$.  It is worth noticing that having $J$ repetitions also solves in part the identifiability problem encountered in \cite{comte2025non} for ${\mathbb E}[b(X_1)]\neq 0$, see the discussion in Section \ref{Identif}. We may underline that the generalization of the deviation underlying the control of the risk of the estimator is not straightforward and requires additional work. In addition to this generalization, we propose an adaptive procedure and we prove a risk bound, which requires the study of a new empirical process. This is also novel for the case $J=1$, for which no adaptive estimator has been examined yet.

The plan of the paper is the following. We first describe in Section \ref{sec::estimator-def} the notation, the projection spaces (models) and the contrast. Our nonparametric estimator is the minimizer on a model of this new contrast. These first elements allow us to discuss the identifiability of Model (\ref{model}). The empirical and theoretical risks are upper bounded in Section \ref{sec::risks}. The first bound expressed with respect to a random distance is quite straightforward. The risk bound expressed in term of the integrated version of the empirical norm requires more technicalities to prove that the two norms are equivalent with probability close to 1 (see Lemma \ref{TheResult} as a step to obtain Theorem \ref{thm::theoretical-risk}). A model selection procedure is defined in Section \ref{Adaptive} and the resulting estimator is proved to be adaptive (see Theorem \ref{Adapt}). It also relies on Lemma \ref{TheResult} but, in addition,  on the study of an empirical process involving dependency (Lemma \ref{lem::bound-emp-process}). Section \ref{Simulations} gives an account of simulation experiments that also provide interesting results and discussions. Most proofs are relegated to Section \ref{proofs}.

\section{Model and minimal contrast estimator}\label{sec::estimator-def}

\subsection{General notation}\label{sec::general-notation}
We consider the model introduced in \eqref{model}. Before defining the estimator of regression function $b$ in the next section, let us introduce some notation. First, define the ${\mathbb R}^n$-valued vectors 
$${\mathbf X}_{\bullet,j}:=(X_{i,j})_{1\leq i\leq n}, \quad {\mathbf Y}_{\bullet,j}:=(Y_{i,j})_{1\leq i\leq n}.$$
Further set
\begin{align*}
    \Xvec := \left( \begin{array}{c}   {\mathbf X}_{\bullet,1} \\ \vdots \\   {\mathbf X}_{\bullet,J}
    \end{array} \right) \in \Rz^{nJ}, \quad   \Yvec := \left( \begin{array}{c}   {\mathbf Y}_{\bullet,1} \\ \vdots \\   {\mathbf Y}_{\bullet,J}
    \end{array} \right) \in \Rz^{nJ}.
\end{align*}
Set ${\mathbf u}_{\bullet,j}:=(\sqrt{\rho}\varepsilon_{0,j} + \sqrt{1-\rho}\varepsilon_{i,j})_{1\leq i\leq n}\in \Rz^{n}$ and
\begin{align*}
    \uvec := \left( \begin{array}{c}   {\mathbf u}_{\bullet,1} \\ \vdots \\   {\mathbf u}_{\bullet,J}
    \end{array} \right).
\end{align*}
Consequently, Model \eqref{model} can also be written as
\begin{align*}
     \Yvec = b(\Xvec) + \uvec.
\end{align*}
Note that the $({\mathbf u}_{\bullet,j})_{1\leq j\leq n}$ are i.i.d. vectors but elements within vectors are not independent. To describe the dependence structure resulting from the common noise component, define 
\begin{align*}
    \operatorname{var}({\mathbf u}_{\bullet,1})/\sigma^2  := {R} := \left( \begin{array}{c c c} 1 & & \rho\\ & \ddots & \\ \rho & & 1   
    \end{array} \right)\in \Rz^{n\times n}.
\end{align*}
Then, for the matrix
\begin{align*}
    \mathbf{R} := \left( \begin{array}{c c c} R & & 0\\ & \ddots & \\ 0 & & R   
    \end{array} \right) \in \Rz^{nJ \times nJ}
\end{align*}
it holds $\operatorname{var}(\uvec)  = \sigma^2 \mathbf{R}$. The eigenvalues both of $R$ and $\mathbf{R}$ are $1-\rho$ and $1 + (n-1)\rho$ with multiplicity $n-1$ and 1 in $R$ respectively. Consequently, for any $\rho\in(0,1)$ both matrices are positive definite. The inverse matrices are given by 
\begin{align*}
     {R}^{-1} := \left( \begin{array}{c c c} \alpha_n & & \beta_n \\ & \ddots & \\ \beta_n & & \alpha_n   
    \end{array} \right), \quad\text{with}\quad \alpha_n := \frac{1+ (n-2)\rho}{(1-\rho)(1+ (n-1)\rho)}, \quad\beta_n := -\frac{\rho}{(1-\rho)(1 + (n-1)\rho)}.
\end{align*} For ${\rm Id}_n$ the $n\times n$ identity matrix and ${\mathbf 1}_n$ the $n$-dimensional vector with all coordinates equal to 1, this matrix can be written as
\begin{equation}\label{InvRsimple}
R^{-1}=(\alpha_n-\beta_n){\rm Id}_n +\beta_n {\mathbf 1}_n {\mathbf 1}_n^\top =\frac  1{1-\rho}\left( {\rm Id}_n - \frac{\rho}{1+(n-1)\rho}{\mathbf 1}_n {\mathbf 1}_n^\top\right). 
\end{equation}
In addition, it  holds that
\begin{align*}
    \Rbig := \left( \begin{array}{c c c} R^{-1} & & 0\\ & \ddots & \\ 0 & & R^{-1}   
    \end{array} \right)\in \Rz^{nJ \times nJ}.
\end{align*}
The corresponding eigenvalues of $R^{-1}$ and $\Rbig$ are
\begin{align*}
    \alpha_n-\beta_n = \frac{1}{1-\rho} \quad\text{and}\quad c_n := \frac{1}{1+(n-1)\rho}.
\end{align*}

\subsection{Minimal contrast estimator}\label{sec::minimal-constrast}
Let us introduce a minimal contrast estimator for the regression function $b\in\mathbb{L}^2(A)$ of Model \eqref{model}, with $A\subset {\mathbb R}$. For this, we consider an orthonormal basis $(\varphi_k)_{k\geq 1}$ of $\mathbb{L}^2(A)$ and associate to the basis a collection of models 
$$S_m:= \text{span}\{\varphi_1,\ldots, \varphi_m\}, \quad m\in {\mathbb N}.$$ 
The idea is to build an estimator $\widehat b_m=\sum_{k=1}^m \widehat a_k \varphi_k$ by considering a generalized least squares contrast given for $t\in S_m$  by
\begin{align*}
    \gamma_{n,J}(t)    &:=\frac{1}{nJ} \left( t(\Xvec)^\top \Rbig t(\Xvec) - 2 \Yvec^\top \Rbig t(\Xvec) \right).
\end{align*} 
First, let us introduce further notation. For $j\in\{1,\ldots,J\}$,  set $\widehat \Phi_{m,j} :=(\varphi_k(X_{i,j}))_{1\leq i\leq n, 1\leq k\leq m}\in {\mathbb R}^{n\times m}$ and define
\begin{align*}
    \wPhi := \left( \begin{array}{c}   \widehat\Phi_{m,1}  \\ \vdots \\    \widehat\Phi_{m,J} 
    \end{array} \right) \in \Rz^{nJ\times m}
\end{align*}
and for $\widehat \Psi_{m,j}^{(R)} := \frac 1n \widehat\Phi_{m,j}^\top R^{-1} \widehat\Phi_{m,j}\in \Rz^{m \times m}$ we set
\begin{align*}
    \wPsi &:= \frac 1J\sum_{j=1}^J \widehat\Psi_{m,j}^{(R)} = \frac{1}{nJ} \wPhi^\top \Rbig \wPhi \in\Rz^{m\times m}.
\end{align*}
The standard minimization of $\gamma_{n,J}$ over $S_m$ leads to the vector of coefficients $\widehat{\textbf a}_{m}=(\widehat{a}_1,\ldots,\widehat{a}_m)\in\Rz^m$ such that
$$\wPsi \widehat{\textbf a}_{m} = \widehat Z_m$$
with $\widehat Z_{m,j}:=\frac 1n \widehat\Phi_{m,j}^\top R^{-1}{\mathbf Y}_{\bullet,j}\in \Rz^{m}$ and 
$$\widehat Z_m:=\frac 1J\sum_{j=1}^J \widehat Z_{m,j} = \frac{1}{nJ} \wPhi^\top \Rbig \Yvec \in \Rz^{m}.$$
Consequently, if $\wPsi$ is invertible, we obtain
\begin{align*}
    \widehat{\textbf a}_{m} =  (\wPsi)^{-1}  \frac{1}{nJ} \wPhi^\top \Rbig \Yvec
\end{align*}
which leads to the proposed estimator $\widehat{b}_m := \sum_{k=1}^m \widehat{a}_k \varphi_k$ of $b$ in the regression model. 

\begin{rem}\label{rem::projection}
Observe that it holds $\wPhi \widehat{\textbf a}_m = \widehat{b}_m(\Xvec)$ and 
\begin{align*}
    P :=\frac{1}{Jn}  \wPhi (\wPsi)^{-1} \wPhi^\top \Rbig \in\Rz^{nJ \times nJ}
\end{align*}
satisfies $P = PP$ and $P^\top \Rbig= \Rbig P$. Consequently, $P$ defines an orthogonal projection on $\Rz^{nJ}$ with respect to the scalar product induced by $\Rbig$.
\end{rem}

\begin{rem} Note that, as more specifically developed in \cite{comte2025non}, using the representation in
(\ref{InvRsimple}), for $n$ large enough, $R^{-1}$ is equivalent to  
$$\frac  1{1-\rho}\left( {\rm Id}_n +  \frac 1n {\mathbf 1}_n^\top {\mathbf 1}_n\right),$$
and since the matrix appears twice in the formula of $\widehat{\textbf a}_{m}$, the factor $1/(1-\rho)$ cancels. Thus, in the definition of    $\widehat{\textbf a}_{m}$ the matrix $R^{-1}$ can be replaced by ${\rm Id}_n +  (1/n) {\mathbf 1}_n^\top {\mathbf 1}_n$, which does not require the knowledge of $\rho.$
\end{rem}

\subsection{About level identifiability}\label{Identif}
Now, we discuss the constraints implicitly contained in the invertibility assumption on $\wPsi$. For this, we first introduce norms adequate for the  model considered in this work.

For $t\in S_m$ with $t= \sum_{k=1}^m a_k \varphi_k$ for $\mathbf{a} = (a_1,\ldots, a_m)^\top \in\Rz^m$, the extended empirical norm of the problem is 
$$\|t\|_{n,J,R}^2 =\frac 1J\sum_{j=1}^J \|t\|_{n,R,j}^2, \quad   
\|t\|_{n,R,j}^2=\frac 1n t({\mathbf X}_{\bullet,j})^\top R^{-1} t({\mathbf X}_{\bullet,j})
= {{\mathbf a}}^\top\widehat\Psi_{m,j}^{(R)} {{\mathbf a}}.$$
It holds that $t(\Xvec) =  \wPhi \mathbf{a}$. Consequently, the empirical norm can also be rewritten as
\begin{align*}
    \|t\|_{n,J,R}^2 = \frac{1}{Jn} t(\Xvec)^\top \Rbig t(\Xvec) = \mathbf{a}^\top \frac{1}{Jn} \wPhi^\top \Rbig \wPhi \mathbf{a} = \mathbf{a}^\top \wPsi \mathbf{a}.
\end{align*}
The corresponding theoretical counterpart is defined as $\Vert t \Vert_{R}^2 := \E[\|t\|_{n,J,R}^2]$. Consequently, setting $\ePsi = \E[\wPsi]$,  it also holds that $$\Vert t \Vert_{R}^2 =\mathbf{a}^\top \ePsi \mathbf{a}.$$ 
Moreover, the following Lemma provides a link between $\wPsi$ and the covariance matrix.
\begin{lem}\label{linkCov}
    For any $t\in S_m$ with $t= \sum_{k=1}^m a_k \varphi_k$ and  $\mathbf{a} = (a_1,\ldots, a_m)^\top \in\Rz^m$, for all $J\geq 1$, it holds that
    \begin{align*}
        \Vert t \Vert_{n,J,R}^2 = \mathbf{a}^\top \wPsi \mathbf{a}  \xrightarrow{n\rightarrow +\infty} \frac{1}{1-\rho} \mathbf{a}^\top \Psi_m \mathbf{a}, \quad\text{almost surely},
    \end{align*}          
    where 
    \begin{equation*}
        \Psi_m := \left(\operatorname{cov}(\varphi_{k_1}(X_{1,1}), \varphi_{k_2}(X_{1,1}))\right)_{1\leq k_1,k_2\leq m}.
    \end{equation*}
\end{lem}
Note that the theoretical objects $\ePsi$, $\Vert \cdot \Vert_{R}$ and $\Psi_m$ are equivalent to the objects of \cite{comte2025non} with the same names. Consequently, results about these objects can be directly used. However, empirical objects depend on $J$ and care must be taken to extend results.

Lemma \ref{linkCov} shows the correspondence between the empirical matrix  $\wPsi$ and $\Psi_m$, which is not the matrix classically considered in nonparametric regression methods\footnote{Indeed, usually the matrix to invert is the cross-moment $\left({\mathbb E}[\varphi_{k_1}(X_{1,1})\varphi_{k_2}(X_{1,1})] \right)_{1\leq k_1, k_2\leq m}$, also denoted by $\Psi_m$ in several papers.}. This covariance matrix can not be invertible if the basis contains a constant function; indeed then, the basis functions are compactly supported,  the support $A$ is chosen such that all $X_{i,j}$ are in $A$ a.s., and thus ${\rm cov}({\mathbf 1}_A(X_{1,1}),\varphi_k(X_{1,1}))\simeq 0$. This enhances a level problem in the estimation procedure: if the constant function is $\varphi_1$ the first function of the basis on $A=[{\tt a}_1, {\tt a}_2]$, i.e. 
$\varphi_1=(1/\sqrt{{\tt a}_2-{\tt a}_1}){\mathbf 1}_{[{\tt a}_1, {\tt a}_2]}$, then the first term of the development of $b$ in the basis is  $$\langle b, \varphi_1\rangle \varphi_1(\,.\,)=  \left( \frac 1{{\tt a}_2-{\tt a}_1}\int_{{\tt a}_1}^{{\tt a}_2}  b(x)dx \right)\,\,  {\mathbf 1}_{[{\tt a}_1, {\tt a}_2]}(\,.\,),$$ and it can not be estimated. This is the reason why in this work we consider bases of $\mathbb{L}^2(A)$ without any constant functions, which makes easy comparisons possible between our setting and the standard context, that is, $\rho=0$. 

Another way of looking at the problem is to consider $Y_{i,j}=\mu + u_{i,j}$; then  we can estimate $\mu$ with 
$$\widehat \mu= \frac 1{nJ}\sum_{i=1}^n\sum_{j=1}^JY_{i,j}.$$ We get an unbiased and  consistent estimator of $\mu$, but we note that the variance of $\widehat \mu$ is of order $1/J$ (equal to $\rho/J+ (1-\rho)/(nJ)$). The level correction is mandatory in the estimation procedure, and we make the transformation $Y_{i,j}-\widehat\mu$ for the estimation of $b^\star=b-\mu$, and then compute errors for $\widehat{b^\star} +\widehat \mu$, leading to a consistency which could not be reached for $J=1$ in \cite{comte2025non}. This is however not completely satisfactory since, as explained above, the correct level correction is rather the mean integral of $b$ on the domain. Further discussions are given in Section \ref{Simulations}.

\section{Risk bounds in empirical and theoretical norms}\label{sec::risks}
The aim is to bound the risk of the estimator proposed in  Section \ref{sec::estimator-def} first in integrated empirical norm ${\mathbb E}[\Vert \,.\, \Vert_{n,J,R}^2]$ and, subsequently, in integrated theoretical norm ${\mathbb E}[\Vert \,.\, \Vert_{R}^2]$. 

\subsection{First bound in empirical norm}

We start by considering the risk of the estimator proposed as expectation of the empirical norm $\Vert \cdot\Vert_{n,J,R}^2$. 

\begin{prop}\label{prop::riskbound-empirical}
    Assume that $\wPsi$ is invertible. The risk bound on $\widehat{b}_m$ expressed as the expectation of the empirical norm $\Vert \cdot\Vert_{n,J,R}$ satisfies
    \begin{align*}
    \E [ \Vert b - \widehat{b}_m \Vert_{n,J,R}^2 ] 
    \leq \inf_{t\in S_m} \Vert b - t \Vert_{R}^2 + \sigma^2\frac{m}{nJ}.
\end{align*}
\end{prop}
We refer to the appendix for the proof. We observe that the risk decomposition is a generalization of the one obtained in Proposition 1 of \cite{comte2025non}, where the variance term $m/n$ for $J=1$ becomes $m/(nJ)$ with only the invertibility constraint on  $\wPsi$.

Since $\mathbf R$ is also a variance matrix with the same eigenvalues as $R$ and its inverse $\Rbig$ is also diagonalizable with the same eigenvalues as $R^{-1}$ (with multiplicities $J(n-1)$ and $J$), we can extend Lemma 1 in \cite{comte2025non} as follows.
\begin{lem}\label{lem::norm-properties}
    The following properties hold:
    \begin{enumerate}
        \item $\Vert t \Vert_{n,J,R}^2\geq 0$ with 
        \begin{align*}
            \frac{1}{1 +(n-1)\rho} \Vert t\Vert^2_{n,J} \leq \Vert t \Vert_{n,J,R}^2 \leq \frac{1}{1-\rho} \Vert t\Vert^2_{n,J}
        \end{align*}
        where $\Vert t\Vert_{n,J}^2 := \frac{1}{nJ} t(\Xvec)^\top t(\Xvec)$.
        \item $\Vert t \Vert_{R}^2\geq 0$ with 
        \begin{align*}
            \frac{1}{1 +(n-1)\rho} \E[t^2(X_{1,1})] \leq \Vert t \Vert_{R}^2 \leq \frac{1}{1-\rho}  \E[t^2(X_{1,1})].
        \end{align*}
    \end{enumerate}
\end{lem}

If $f$ is bounded and $b_m:=\sum_{j=1}^m \langle b, \varphi_j\rangle \varphi_j$ denotes the orthogonal projection of $b$ on $S_m$, then, by Lemma \ref{lem::norm-properties}, 
\begin{align*}
    \inf_{t\in S_m} \Vert b - t \Vert_{R}^2 \leq \frac{1}{1-\rho}  \inf_{t\in S_m} \Vert b - t \Vert_{R}^2 \leq \frac{\Vert f\Vert_\infty}{1-\rho} \Vert b_m - b\Vert^2,
\end{align*} 
where $\Vert \cdot \Vert$ denotes the integral norm in $\mathbb{L}^2(\mathbb{R})$. Clearly $\Vert b_m - b\Vert^2$ tends to zero when $m$ increases. In contrast, the variance term $m/nJ$ increases with $m$. Thus, Proposition \ref{prop::riskbound-empirical} exhibits a standard bias/variance decomposition requiring a compromise, and is associated with a sample size $nJ$ without loss despite the dependency introduced by the common noise. Clearly, under standard regularity assumptions related to a regularity order for $b$, a nonparametric rate of order of a negative power of $nJ$ can be expected: this explains why a level correction with risk of order $1/J$ is not completely satisfactory, see Section \ref{Identif}.

\subsection{Risk bound in theoretical norm}

To obtain a result with respect to the risk in theoretical norm, we have to compare precisely the empirical norm to its expectation. We give in Lemma \ref{Lemm2} hereafter the decomposition of the empirical norm,  involving several processes, whose deviation have to be controlled in the following.  
\begin{lem}\label{Lemm2}
    For the empirical norm the  following decomposition holds.
    \begin{align*}
        \Vert t \Vert_{n,J,R}^2 = \alpha_n Z_{n,J}(t) + c_n(\E[t(X_{1,1})])^2 + 2 c_n \E[t(X_{1,1})] V_{n,J}(t) + (n-1)\beta_n U_{n,J}(t)
        \end{align*}
        with
        \begin{align}
            Z_{n,J}(t)&:= \frac 1J\sum_{j=1}^J Z_{n,j}(t), \quad Z_{n,j}(t)=\frac{1}{n}\sum_{i=1}^n (t(X_{i,j})-\E[t(X_{1,1})])^2\label{eq::zJ}\\
            V_{n,J}(t) &:= \frac{1}{J}\sum_{j=1}^JV_{n,j}(t), \quad V_{n,j}(t)=\frac 1n\sum_{i=1}^n (t(X_{i,j})-\E[t(X_{1,1})])\nonumber\\
            U_{n,J}(t) &:= \frac{1}{J} \sum_{j=1}^J U_{n,j}(t), \quad U_{n,j}(t) = \frac 1{n(n-1)} \sum_{1\leq i\not= k \leq n} (t(X_{i,j})-\E[t(X_{1,1})])(t(X_{k,j})-\E[t(X_{1,1})]).\label{eq::uJ}
        \end{align}
 \end{lem}
 The main processes of the decomposition of $\|t\|_{n,J,R}^2$ are:  $Z_{n,J}(t)$, which is studied thanks to a matrix Chernov deviation inequality (Lemma \ref{lem::tropp-chernoff} from \cite{tropp2012user} recalled in the Appendix); and $U_{n,J}(t)$ which is a degenerated $U$-statistics of second order and is treated with the corresponding deviation inequality stated in  \cite{GineNickl2016Book} (Theorem 3.4.8 p. 183; see also Theorem 3.4 in \cite{HoudreReynaud2003Ustat}). 
As $V_{n,J}(t)$ and $U_{n,J}(t)$ are centered, it straightforwardly holds that
$$\|t\|_R^2= {\mathbb E}[\|t\|_{n,J,R}]= \alpha_n {\rm Var}(t(X_{1,1})) + c_n(\E[t(X_{1,1})])^2.$$
As $\alpha_n\rightarrow_{n\rightarrow +\infty} 1/(1-\rho)$ and $c_n=1/[1+(n-1)\rho]\rightarrow_{n\rightarrow +\infty} 0$, we find, in accordance with Lemma \ref{linkCov}, that
$$  \lim_{n\rightarrow +\infty} \|t\|_R^2 = \frac 1{1-\rho}  {\rm Var}(t(X_{1,1}).$$

We note that the following equality holds:
\begin{align*}
    \Vert t \Vert^2_{n,J,R} = \frac{1}{1-\rho} \frac{1}{nJ}\sum_{j=1}^J \sum_{i=1}^n \left[t(X_{i,j})  - \frac{1}{n}\sum_{k=1}^n t(X_{k,j}) \right]^2 + c_n \frac{1}{J}\sum_{j=1}^J \left(\frac{1}{n}\sum_{i=1}^n t(X_{i,j}) \right)^2.
\end{align*}
Therefore, $ \Vert t \Vert^2_{n,J,R} =0$ implies $t(X_{i,j})=0$ a.s. for all $i,j$. For some orthonormal bases, such as the Hermite or Tchebychev basis, $ \Vert t \Vert^2_{n,J,R} =0$ for $t\in S_m$ even implies that $t=0$ for $nJ$ larger than $m$, see Section \ref{Simulations} for details on the bases. \\

Next, the aim is to handle the integrated risk with respect to the squared norm $\Vert \cdot \Vert_R^2$. As discussed in \cite{comte2025non}, this task is far from simple and the extension to the more general setting of this , i.e.  $J>1$, also requires an adaptation of most steps. To bound the risk in theoretical norm, the following assumption is necessary.
\begin{ass}\label{ass::theoretical-bound}
\begin{enumerate}
    \item[(i)] Assume that for
    \begin{align*}
        L(m) := \sup_{x\in\Rz} \sum_{k=0}^m \varphi_k^2 (x) < +\infty
    \end{align*}
     it holds $L(m)=m^s$ and the functions $\varphi_j$ are such that $\sup_{x\in A}|\varphi_j(x)|\leq \theta$, for some $\theta, s>0$.
    \item[(ii)] For $m, n, J\in\Nz$ and some constant $c^\star>0$ assume that the following stability condition holds:
\begin{equation}\label{Stability}
    (m\vee L(m)) (\Vert \Psi_m^{-1} \Vert_{\operatorname{op}} \vee 1) \leq c^\star \frac{n\sqrt{J}}{\log(n)\sqrt{1\vee\log(J)}},
\end{equation}
using the notation $x\vee y = \min(x,y)$ for any $x,y\in\mathbb{R}$.
\end{enumerate}
\end{ass}

\begin{rem}
    The results of this work also hold up to constants if there is an additional constant in the assumption on $L$, i.e. $L(m)=cm^s$ for some $c>0$. Note that combining the stability condition \eqref{Stability} and Assumption \ref{ass::theoretical-bound} (i) results in $m^{s\vee 1} \leq n\sqrt{J}$.
    In addition, we point out that in contrast to the model without repeated measurements, discussed in \cite{comte2025non}, the stability condition is posed on the covariance $\Psi_m$  instead of $\ePsi$. The next result states that the stability condition for $\ePsi$ follows.
 \end{rem}

\begin{prop}\label{prop::new-prop3}
    Under Assumption \ref{ass::theoretical-bound}, we have that for $n\geq 3 $
   \begin{equation}\label{StabilityR}
        L(m) (\Vert (\ePsi)^{-1} \Vert_{\operatorname{op}} \vee 1) \leq  c^{\star \star} \frac{n\sqrt{J}}{\log(n)\sqrt{1\vee\log(J)}}, \quad c^{\star \star}=2(1-\rho) c^\star.
    \end{equation}
\end{prop}

We refer to section \ref{sec:proofprop2} for the proof. To control the theoretical norm of the estimator, we need to impose the stability on the estimator. Thus, the corresponding empirical version of the stability condition \eqref{StabilityR} is given by
\begin{align*}
    \Lambda_m := \left\{ L(m) (\Vert (\wPsi)^{-1} \Vert_{\text{op}} \vee 1) \leq 4 c^{**} \frac{n \sqrt{J}}{\log(n)\sqrt{1\vee\log(J)}}\right\}.
\end{align*}
Let us define now the restricted estimator $$\widetilde{b}_m := \mathbf{1}_{\Lambda_m}\widehat{b}_m.$$

Further, we define the following constants. For $p\geq 1$ and $r\geq 2$ and $c_0 = 18.6$  define
\begin{align}
    {\mathfrak c}_r := 6r^2 c_0 (1+\theta^2) 8\sqrt{6} \quad \text{ and }\quad{\mathbf c}_p^\star := \frac{1-3\log(3/2)}{8(p+1)}.\label{eq::constants-def}
\end{align}
In the results of this work, these constants restrict the constant $c^\star$ of stability condition \eqref{Stability} for different choices of $p$ and $r$, see Remark \ref{rem::constants} and Lemma \ref{TheResult} below for more details.  Then, the following result holds for the risk bound in term of the expectation of the theoretical norm.

\begin{thm}\label{thm::theoretical-risk}
Let Assumption \ref{ass::theoretical-bound} be satisfied. For ${\mathfrak c}_r$ and ${\mathbf c}_p^\star$ defined in \eqref{eq::constants-def} with $p=4$ and $r=4 + 2\min(1/s,1)$, assume that for the constant $c^\star$ of stability condition \eqref{Stability} holds
 \begin{equation}\label{cstar}
     c^\star \leq \frac 1{8{\mathfrak c}_r} \wedge \frac{{\mathbf c}_p^\star}3.
 \end{equation} 
 we have for  $n\geq 192/\rho$ , $J \leq n(n-1)/\log^2(n)$ and $m\leq nJ$ that 
    \begin{align*}
        \E [ \Vert b - \widetilde{b}_m \Vert_{R}^2 ] 
    \leq C_1 \left(\inf_{t\in S_m} \Vert b - t \Vert_{R}^2 + \sigma^2\frac{m}{nJ}\right) +  \frac {C_2}{nJ} 
    \end{align*}
    for $C_1$ a numerical constant ($C_1=34$ suits) and $C_2$ is a constant depending on ${\mathfrak c}_r $, ${\mathbf c}_p^\star$, $\rho$, $\|b\|_R^2$, ${\mathbb E}[Y_{1,1}^4]$, but not on $m, n, J$.
\end{thm}

Let us mention that the condition $J \leq n(n-1)/\log^2(n)$ is analogous to the constraint $\sqrt{J}/n \rightarrow  0$ set in \cite{GL2026} for asymptotic normality results.

\begin{rem}\label{rem::constants}
    Note that depending on the choices of parameters $p$ and $r$, the condition (\ref{cstar}) on the constant $c^\star$ changes. More precisely, the larger $p$ and $r$, the smaller $c^\star$, leading to a more restrictive stability condition \eqref{Stability}. These constants appear due to the application of Lemma \ref{TheResult}, see below, to show the result for Theorem \ref{thm::theoretical-risk} with  $p=4$ and $r=4 + 2\min(1/s,1)$. For the adaptive result Theorem \ref{Adapt} in the next section, we apply the same Lemma for $p=  5  $ and $r=5 + 2\min(1/s,1)$, leading to a change in constants.
    We also note that, in comparison to \cite{comte2025non}, there is a change in the constant $c_0$ because we applied the deviation inequality for $U$-statistics of \cite{GineNickl2016Book} instead of the one of \cite{HoudreReynaud2003Ustat}. This is due to the fact that the variables are now $J$-dimensional vectors instead of real random variables for $J=1$ in \cite{comte2025non}.
\end{rem}

We provide a proof of Theorem \ref{thm::theoretical-risk} showing the main steps of the proof which relies on Lemma \ref{TheResult}; this Lemma is rather difficult to demonstrate and its proof is referred to Section \ref{proofs}.

\begin{proof}[Proof of Theorem \ref{thm::theoretical-risk}]
The first main idea for the proof of Theorem \ref{thm::theoretical-risk} is to decompose the risk on the set $\Omega_m$ on which the empirical norm and the theoretical norm are equivalent, i.e. 
\begin{align*}
    \Omega_m := \left\{\forall t\in S_m, t\not= 0, \left\vert \frac{\Vert t\Vert^2_{n,J,R}}{\Vert t\Vert_{R}^2} - 1 \right\vert \leq \frac{3}{4} \right\}=\left\{ \left\| (\Psi_m^{(R)})^{-1/2} \widehat \Psi_m^{(R)}  (\Psi_m^{(R)})^{-1/2}- {\rm Id}_m  \right\|_{{\rm op}} \leq \frac 34\right\}.
\end{align*}
and its complement $\Omega_m^c$. The main task, here is to bound the probability of the complement, and therefore, prove the following Lemma.
\begin{lem}\label{TheResult}
Let Assumption \ref{ass::theoretical-bound} be satisfied. For ${\mathfrak c}_r$ and ${\mathbf c}_p^\star$ defined in  \eqref{eq::constants-def} with arbitrary $p\geq 1$ and $r\geq 2$, assume that for the constant $c^\star$ of stability condition \eqref{Stability} holds
$$c^\star \leq \frac 1{8{\mathfrak c}_r} \wedge \frac{{\mathbf c}_p^\star}3.$$
Then, we have for $n\geq 192/\rho$ and $J \leq n(n-1)/\log^2(n)$ that
$${\mathbb P}(\Omega_m^c)\leq 12 (nJ)^{-\min( p , r-2\min(1/s,1))}.$$
\end{lem}

For the other points of the proof, we follow the line of the proof of Theorem 2 in \cite{comte2025non}, which is standard. More precisely, we get
\begin{align}\label{eq::risk-bound-step}
{\mathbb E}\left[\| \widetilde b_m-b\|^2_R\right] \leq 34 \inf_{t\in S_m} \|t-b\|_R^2 + 16 \sigma^2 \frac{m}{n  J} + 4 {\mathbb E}[\|\widehat b_m\|_R^2 {\mathbf 1}_{\Lambda_m} {\mathbf 1}_{\Omega_m^c} ] + \|b\|_R^2 (4{\mathbb P}(\Omega_m^c) + {\mathbb P}(\Lambda_m^c)).
\end{align}
To control the third summand, we first see that
$$\|\widehat{b}_m\|_R^2 = \frac 1{(nJ)^2} \Yvec^\top {\mathbf R}^{-1} \widehat \Phi_m (\widehat \Psi_m^{(R)})^{-1} \Psi_m^{(R)} (\widehat \Psi_m^{(R)})^{-1} \widehat\Phi_m^\top {\mathbf R}^{-1} \Yvec.$$ Next, we use that on $\Lambda_m$, we have
$$  \Vert (\wPsi)^{-1} \Vert_{\text{op}}^2  \leq 16 (c^{**})^2  \left(\frac{n \sqrt{J}}{L(m) \log(n)\sqrt{1\vee\log(J)}}\right)^2,$$
and that $\|{\mathbf R}^{-1}\|_{\rm op} \leq 1/(1-\rho)$, $\| \Psi_m^{(R)}\|_{\rm op}\leq L(m)/(1-\rho)$ and $\|\widehat\Phi_m^\top \widehat \Phi_m\|_{\rm op}\leq nJ L(m)$.  We get
$$\|\widehat{b}_m\|_R^2 \leq 64 \frac{(c^\star)^2}{(1-\rho)^3} \frac 1{\log^2(n) (1\vee \log(J))} \|\Yvec \|^2. $$
Thus
\begin{eqnarray*} {\mathbb E}[\|\widehat b_m\|_R^2 {\mathbf 1}_{\Lambda_m} {\mathbf 1}_{\Omega_m^c} ] &\leq & 
64 \frac{(c^\star)^2}{(1-\rho)^3} \frac 1{\log^2(n) (1\vee \log(J))} {\mathbb E}[\|\Yvec \|^2{\mathbf 1}_{\Omega_m^c} ] \\ &\leq & 64 \frac{(c^\star)^2 {\mathbb E}^{1/2}[Y_{1,1}^4] }{1-\rho} \frac{nJ}{\log^2(n) (1\vee \log(J))} {\mathbb P}^{1/2}(\Omega_m^c).
\end{eqnarray*}
Further, we have $\Lambda_m^c\subseteq \Omega_m^c$ analogously to Proposition 4 in \cite{comte2025non}. 
Then, we apply Lemma \ref{TheResult} with $r=4 + 2\min(1/s,1)$ and $p=4$ to obtain 
\begin{align*}
    {\mathbb P}^{1/2}(\Omega_m^c) \leq \frac{\sqrt{12}}{(nJ)^2}.
\end{align*}
Plugging this into \eqref{eq::risk-bound-step}, the result follows, i.e.
\begin{align*}
    {\mathbb E}\left[\| \widetilde b_m-b\|^2_R\right] \leq 34 \left(\inf_{t\in S_m} \|t-b\|_R^2 +  \sigma^2 \frac{m}{n  J} \right) +\sqrt{12}\left(256 \frac{(c^\star)^2 {\mathbb E}^{1/2}[Y_{1,1}^4] }{1-\rho} + 8\|b\|_R^2\right) \frac{1}{nJ}.
\end{align*}
\end{proof}

\section{Adaptive Result}\label{Adaptive}

The goal now is to propose an adaptive estimator and provide upper bounds on the corresponding quadratic risk. The result is new, even for $J=1$ and requires to solve the dependency problem between the $u_{i,j}$'s when bounding the deviations of a key empirical process. For this, we define the collection of models
\begin{align*}
    \widehat{\mathcal{M}}_{n,J} := \left\{ m\in\mathbb{N}, m\leq nJ: L(m) (\Vert (\widehat{\Psi}_m^{(R)})^{-1}\Vert_{\text{op}}\vee 1) \leq \frac{4c^{**} n\sqrt{J}}{\log (n) \sqrt{\log(J)\vee 1}}\right\},
\end{align*}
and we set
\begin{align*}
    \widehat{m} := \arg\min_{m\in\widehat{\mathcal{M}}_{n,J}} \left\{ - \Vert \widehat{b}_m \Vert_{n,J,R}^2 + {\rm pen}(m) \right\}, \quad {\rm pen}(m)= \kappa \sigma^2 \frac{m}{nJ},
\end{align*}
where $\kappa$ is a numerical constant.
A theoretical counterpart for $ \widehat{\mathcal{M}}_{n,J}$ is defined by
\begin{align*}
    \mathcal{M}_{n,J} := \left\{ m\in\mathbb{N}, m\leq nJ: L(m)  (\Vert ({\Psi}_m^{(R)})^{-1}\Vert_{\text{op}}\vee 1) \leq \frac{c^{**} n\sqrt{J}}{\log (n) \sqrt{\log(J)\vee 1} }\right\}.
\end{align*}

We introduce the following ``light tail" assumption on the noise $\varepsilon_{1,1}$.
\begin{ass}\label{subExp}
    There exist  $c_1, c_2>0$, such that  for all $ K>0$ it holds that ${\mathbb P}(|\varepsilon_{1,1}|>K)\leq c_1e^{-c_2K}$.
\end{ass}
\begin{rem} We may have considered the following sub-Gaussian assumption (see \cite{BCVEsaim2001}):
\begin{equation}\label{HypSubGauss} \forall u \in {\mathbb R}, {\mathbb E}(e^{u\varepsilon_{1,1}})\leq e^{u^2s^2/2}. 
\end{equation}
This assumption is satisfied for $\varepsilon_{1,1}$  bounded by $B$ (with $s=B$) and Gaussian  $\varepsilon_{1,1}$ with $s^2={\rm Var}(\varepsilon_{1,1})$.
It is worth noting that (\ref{HypSubGauss}) implies 
\begin{equation}\label{MajProb} {\mathbb P}(|\varepsilon_{1,1}|> K) \leq 2 e^{-K^2/(2s^2)}.
\end{equation}
 See Section \ref{sec::equ-proof} for the proof of (\ref{MajProb}). Thus, Assumption \ref{subExp} is weaker and still allows us to capture bounded and Gaussian variables, but also exponential and Gamma mixtures. However, $\log^3(nJ)$ in the last term of the Inequality of Theorem \ref{Adapt} under Assumption \ref{subExp} would only be $\log^{3/2}(nJ)$ under (\ref{HypSubGauss}).\end{rem}

Under this additional assumption, we prove the following result for the proposed adaptive estimator of the regression function in empirical norm.

\begin{thm}\label{Adapt} Assume that the Assumptions of Theorem \ref{thm::theoretical-risk}  are fulfilled with $p=  5  $ and $r=5 + 2\min(1/s,1)$. Let Assumption \ref{subExp} be satisfied and ${\mathbb E}[b^4(X_{1,1})]<+\infty$. 
Then, for $n\geq n_0:=192/\rho$, there exists $\kappa_0>0$ such that for any $\kappa\geq \kappa_0$, 
$${\mathbb E}(\|\widehat b_{\widehat m} -b\|_{n,J,R}^2) \leq C\inf_{m\in {\mathcal M}_{n,J}} \left\{ 
 \inf_{t\in S_m} \|t-b\|_R^2 +  \kappa \sigma^2 \frac{m}{nJ}\right\} + C' \frac{\log^3(nJ)}{nJ}.$$
 where $C$ is a numerical constant ($C=8$ suits) and $C'$ a constant depending on the moments of $b(X_{1,1})$ and $\varepsilon_{1,1}$. The value $\kappa_0=64$ suits.
\end{thm}

The result of Theorem \ref{Adapt} states that the estimator $\widehat b_{\widehat m}$ automatically reaches the squared bias/variance compromise in the collection ${\mathcal M}_{n,J}$, up to the multiplicative constant $C$ (the closer to 1 the better) and the negligible additive residuals of order $\log^3(nJ)/(nJ)$. In practice, the implementation of the procedure requires a fixed value for $\kappa$, which is done from preliminary simulation experiments. This question appears in the theory of model selection by penalization since the beginning, see \cite{BBM99}, and concrete methods for doing this are given in \cite{Baudry12}.\\

We give the main arguments of the proof of Theorem \ref{Adapt}, which is quite complex and we postpone the proofs of auxiliary results to Section \ref{sec::proofs-lemmas-adaptive}.

\begin{proof}[Proof of Theorem \ref{Adapt}] 
    Let us define
\begin{align*}
    \mathcal{M}_{n,J}^+ := \left\{ m\in\mathbb{N}, m\leq nJ, \;  L(m)  (\Vert ({\Psi}_m^{(R)})^{-1}\Vert_{\text{op}}\vee 1) \leq \frac{7 c^{\star\star} n\sqrt{J}}{\log (n) \sqrt{\log(J)\vee 1} }\right\}
\end{align*}
and 
\begin{align}
\Omega_{n,J} :=\cap_{m\in {\mathcal M}_{n,J}^+} \Omega_m.\label{eq::Omega-union}
\end{align}

We now split the risk on the set $\Omega_{n,J}$ to get
\begin{align}
    {\mathbb E}[\|\widehat b_{\widehat m} -b\|_{n,J,R}^2] = {\mathbb E}[\|\widehat b_{\widehat m} -b\|_{n,J,R}^2{\mathbf 1}_{\Omega_{n,J}}]  + {\mathbb E}[\|\widehat b_{\widehat m} -b\|_{n,J,R}^2{\mathbf 1}_{\Omega_{n,J}^c}].\label{eq::omega-risk-decomp}
\end{align}
To bound the risk of the estimator on $\Omega_{n,J}$, we also define the following set
\begin{align}
    \Xi_{n,J} = \{ \mathcal{M}_{n,J} \subset \widehat{\mathcal{M}}_{n,J} \subset \mathcal{M}_{n,J}^+\}.\label{eq::Xiset}
\end{align}
\begin{lem}\label{OmegaXi}
Under stability condition \eqref{Stability} it holds for $\Omega_{n,J}$ and $\Xi_{n,J}$ defined in \eqref{eq::Omega-union} and \eqref{eq::Xiset}, respectively, that 
\begin{displaymath}
\Omega_{n,J}\subset\Xi_{n,J} .
\end{displaymath}
\end{lem}
\noindent For the contrast, it holds 
\begin{align*}
    \gamma_{n,J}(t) = \Vert b-t \Vert_{n,J,R}^2 - \Vert b\Vert_{n,J,R}^2 - 2\nu_{n,J} (t), \quad \nu_{n,J} (t) = \frac{1}{nJ} \uvec^\top \mathbf{R}^{-1} t(\Xvec). 
\end{align*}
Therefore, writing that for all $ m\in \widehat{\mathcal M}_{n,J}$ and $t\in S_m$ 
$$\gamma_{n,J}(\widehat b_{\widehat m})+ {\rm pen}(\widehat m) \leq \gamma_{n,J}(t) + {\rm pen}(m)$$
implies
$$\|\widehat b_{\widehat m}-b\|_{n,J,R}^2 \leq \|t-b\|_{n,J,R}^2 + {\rm pen}(m) + 2 \nu_{n,J}(\widehat b_{\widehat m}-t) -{\rm pen}(\widehat m).$$

Now, we work on $\Omega_{n,J}$, on which it holds by Lemma \ref{OmegaXi} that $\mathcal{M}_{n,J} \subset \widehat{\mathcal{M}}_{n,J}$. Consequently, the inequality holds for any $m\in {\mathcal M}_{n,J}$, and we  choose $t=\widehat b_m$.  We obtain, on $\Omega_{n,J}$ and for all $ m\in {\mathcal M}_{n,J}$
\begin{eqnarray}\nonumber  \|\widehat b_{\widehat m}-b\|_{n,J,R}^2 &\leq &\|\widehat b_m-b\|_{n,J,R}^2 + {\rm pen}(m) +  2 \nu_{n,J}(\widehat b_{\widehat m}-\widehat b_m) -{\rm pen}(\widehat m) \\ \nonumber &\leq & 
\|\widehat b_m-b\|_{n,J,R}^2 + {\rm pen}(m) + \frac 14 \|\widehat b_{\widehat m}-\widehat b_m\|_{n,J,R}^2 \\\label{DecompOmega}   && +  4 \left(\sup_{t\in S_m+S_{\widehat m}, \|t\|_{n,J,R}=1} \nu_{n,J}^2(t) - p(m,\widehat m) \right)_+  
+ 4 p(m, \widehat m)- {\rm pen}(\widehat m)
\end{eqnarray}
with 
\begin{equation}\label{defpm}
   p(m,m') := 8 \sigma^2 \frac{m\vee m'}{nJ}. 
\end{equation} 
Next, note that for {$\kappa \geq 32$} it holds that $ 4 p(m, \widehat m) \leq {\rm pen}(m) + {\rm pen}(\widehat m)$. Further, we have that
\begin{align*}
    \frac 14 \|\widehat b_{\widehat m}-\widehat b_m\|_{n,J,R}^2 \leq \frac 12 \|\widehat b_{\widehat m}- b\|_{n,J,R}^2 + \frac 12 \|b -\widehat b_m\|_{n,J,R}^2.
\end{align*}
Plugging both into \eqref{DecompOmega} and reorganizing the terms yields for $m\in\mathcal{M}_{n,J}$ on $\Omega_{n,J}$ that
\begin{align*}
     \|\widehat b_{\widehat m}-b\|_{n,J,R}^2 &\leq 3  \|\widehat b_{ m}- b\|_{n,J,R}^2 + 4 {\rm pen}(m) + 8 \left(\sup_{t\in S_m+S_{\widehat m}, \|t\|_{n,J,R}=1} \nu_{n,J}^2(t) - p(m,\widehat m) \right)_+ .
\end{align*}
Plugging this in turn into \eqref{eq::omega-risk-decomp}, we obtain for any $m\in\mathcal{M}_{n,J}$
\begin{align}
      {\mathbb E}[\|\widehat b_{\widehat m} -b\|_{n,J,R}^2] &\leq {\mathbb E}[\|\widehat b_{\widehat m} -b\|_{n,J,R}^2{\mathbf 1}_{\Omega_{n,J}^c}] + 3 {\mathbb E} [ \|\widehat b_{ m}- b\|_{n,J,R}^2] + 4 {\rm pen}(m) \nonumber\\
      &\qquad + 8 {\mathbb E} \left[\left(\sup_{t\in S_m+S_{\widehat m}, \|t\|_{n,J,R}=1} \nu_{n,J}^2(t) - p(m,\widehat m) \right)_+ {\mathbf 1}_{\Omega_{n,J}} \right].\label{eq::decomp-on-omega2}
\end{align}
We obtain the following two Lemmas to control the first and the last summand, see Section \ref{sec::proofs-lemmas-adaptive} for the proofs.
\begin{lem}\label{lem::bound-on-complement}
Let Assumption \ref{ass::theoretical-bound} be satisfied. For ${\mathfrak c}_r$ and ${\mathbf c}_p^\star$ defined in \eqref{eq::constants-def} with $p=  5 $ and $r=5 + 2\min(1/s,1)$, assume that for the constant $c^\star$ of stability condition \eqref{Stability} holds $c^\star \leq 1/(8{\mathfrak c}_r) \wedge {\mathbf c}_p^\star /3$. 
Further, assume that $\E[b^4(X_{1,1})], \E[Y_{1,1}^4] <+\infty$.
Then, for $n\geq 192/\rho $ and $J \leq n(n-1)/\log^2(n)$,  we have
    \begin{align*}
        {\mathbb E}[\|\widehat b_{\widehat m} -b\|_{n,J,R}^2{\mathbf 1}_{\Omega_{n,J}^c}] \leq \frac{C}{nJ}
    \end{align*}
    for a constant $C$  depending on $\rho$, $c^{\star\star}$, $s$, $\E[Y_{1,1}^4]$ and $\E[b^4(X_{1,1})]$.
\end{lem}
\begin{lem}\label{lem::bound-emp-process}
Under Assumption \ref{subExp},  for $p(m,m')$ defined by (\ref{defpm}) and any $m\in\mathcal{M}_{n,J}$ it holds that
\begin{align*}
     {\mathbb E} \left[\left(\sup_{t\in S_m+S_{\widehat m}, \|t\|_{n,J,R}=1} \nu_{n,J}^2(t) - p(m,\widehat m) \right)_+ {\mathbf 1}_{\Omega_{n,J}} \right] \leq C\frac{\log^3(nJ)}{nJ}
\end{align*}
for a constant $C$ depending on $c_1, c_2$ (the constants of Assumption \ref{subExp}), $c^{\star \star}$, $\rho$ and ${\mathbb E}(Y_{1,1}^4)$.
\end{lem}
\noindent Finally, applying Lemma \ref{lem::bound-on-complement} and \ref{lem::bound-emp-process} and Proposition \ref{prop::riskbound-empirical} on  \eqref{eq::decomp-on-omega2} we obtain
\begin{align*}
  {\mathbb E}[\|\widehat b_{\widehat m} -b\|_{n,J,R}^2] \leq C\inf_{m\in {\mathcal M}_{n,J}} \left\{ 
 \inf_{t\in S_m} \|t-b\|_R^2 +  \kappa \sigma^2 \frac{m}{nJ}\right\} + C' \frac{\log^2(nJ)}{nJ}
\end{align*}
for $C=8$ and $C'$ a constant depending on the moments of $b(X_{1,1})$ and $\varepsilon_{1,1}$, which yields the result.
\end{proof}

\section{Simulation experiments}\label{Simulations}

In this section, we illustrate the theoretical results of this work with some simulation experiments. For this, we consider two bases, which are complete orthonormal systems, containing no constant function (see Section \ref{Identif} for the connection to identifiablity):
\begin{itemize}
    \item A Hermite basis, with support ${\mathbb R}$, $(h_j(x)=c_jH_j(x) e^{-x^2/2})_{j\geq 0}$ where $H_j$ is the $j^{{\rm th}}$ Hermite polynomial as in \cite{comte2025non}, see therein.
    \item A Tchebychev basis $(u_j(x))_{j\geq 0}$, defined below, with compact support $[-1,1]$, or $(v_j(x))_{j\geq 0}$ with compact support $[{\tt a}_1,{\tt a}_2]$, ${\tt a}_1<{\tt a}_2$, ${\tt a}_1, {\tt a}_2 \in {\mathbb R}$.
\end{itemize}

Tchebychev polynomials of the first kind are recursively defined for $x\in\mathbb{R}$ by
$$U_0(x)=1, U_1(x)=2x, \quad U_n(x)=2xU_{n-1}(x)- U_{n-1}(x), n\geq 2.$$
They are orthogonal with respect to the scalar product in ${\mathbb L}^2([-1,1], (1-x^2)^{1/2}dx)$ and have squared ${\mathbb L}^2$-norm in this context equal to $\pi/2$. So, we can consider the orthonormal ${\mathbb L}^2([-1,1],dx)$-basis:
$$u_j(x)= \sqrt{\frac 2\pi} U_j(x) (1-x^2)^{1/4}{\mathbf 1}_{[-1,1]}(x).$$
The polynomials are such that $\sup_{x\in [-1,1]} U_j(x)=j+1$, with $U_j(1)=j+1$. So, we get 
$$\sup_{x\in [-1,1]}\sum_{j=0}^{m-1} u_j^2(x) \leq \frac 2{\pi}  \sum_{j=0}^{m-1} (j+1)^2=m(m+1)(2m+1)/(3\pi):=L(m).$$
We also set on $[{\tt a}_1, {\tt a}_2]$, ${\tt a}_1<{\tt a}_2$,
$$v_j(x)=\frac{2}{{\tt a}_2-{\tt a}_1} u_j\left(\frac{2x-({\tt a}_1+{\tt a}_2)}{{\tt a}_2-{\tt a}_1}\right){\mathbf 1}_{[{\tt a}_1,{\tt a}_2]}(x).$$

For the regression function $b$, we consider the examples of \cite{comte2025non},  four examples of odd functions
\begin{equation}\label{defbi} b_1(x)=2x, \quad b_2(x)=3\sin(0.8\, \pi x), \quad b_3(x)= \frac{4x}{1+x^2}, \quad b_4(x)=1.75 x^3\exp(-0.5 |x|)
\end{equation}
and also some even ones
\begin{equation}\label{defbsuite} b_5(x)=3.25|x|,  \quad b_6(x)=4x^2\exp(-x^2/4), \quad b_7(x)= 5.5\sin^2(2x),
\end{equation}
where the multiplicative constants are chosen so that for $X$ following a ${\mathcal N}(0,1)$ distribution,  the signal to noise ratio (snr) is of approximately 2. The snr is measured as the empirical standard deviation of $b_k(X)$  divided by the standard deviation of the noise. We consider the two cases that the $X_{i,j}$'s follow either a ${\mathcal N}(0,1)$ distribution (which is symmetric, see Table \ref{tab1})  or a uniform density $2\sqrt{3}({\mathcal U}([0,1])-0.25)$ (see Table \ref{tab2}).

\begin{table}
 \hspace{-1cm}\footnotesize{
 \begin{tabular}{cc||cccc|cccc||cccc|cccc}
& & \multicolumn{8}{c||}{Hermite basis} & \multicolumn{8}{c}{Tchebychev basis}\\
&& \multicolumn{4}{c}{$n=250$} & \multicolumn{4}{c||}{$n=500$} & \multicolumn{4}{c}{$n=250$}  & \multicolumn{4}{c}{$n=500$} \\
&& \multicolumn{2}{c}{$J=20$} &  \multicolumn{2}{c|}{$J=40$} &  \multicolumn{2}{c}{$J=20$} &  \multicolumn{2}{c||}{$J=40$} & \multicolumn{2}{c}{$J=20$} &  \multicolumn{2}{c|}{$J=40$} &  \multicolumn{2}{c}{$J=20$} &  \multicolumn{2}{c}{$J=40$} \\
&& Est & Or & Est & Or &  Est & Or & Est & Or & Est & Or & Est & Or & Est & Or & Est & Or \\  \hline
       &MSE & 1.96  & 1.86  & 0.99  & 0.92 & 1.95 & 1.88 & 0.93 & 0.89 & 2.77 & 2.31 & 1.48& 1.2  &2.40  & 2.18 &  1.21 & 1.08   \\
$b_1$  &  std  & 2.37  & 2.34 &  1.03 & 1.00 & 2.38 & 2.33 &  1.23 & 1.21 & 2.37 & 2.32&  1.07 & 1.02 & 2.36  & 2.34 & 1.21 & 1.20  \\
       &dim      &  13.4 & 13.6 &  14.1 & 14.3 & 14.5 & 15.0 & 15.1 & 15.8 & 15.3 & 15.4  & 15.1 & 18.0 & 19.0 & 19.4 & 19.4 & 22.1   \\
       &         &   &  &   &  &  &  &  &  &  & & &  &  &  & &    \\
       &MSE &  3.24 & 3.02 &  1.89 & 1.76 &  3.04 & 2.93 & 1.58 & 1.52 & 3.80 & 3.40 & 2.14 & 1.99 & 3.32 & 3.12  & 1.72 & 1.63 \\
 $b_2$&std   & 3.35  & 3.27 &  2.11 & 2.06 & 3.80 & 3.76 & 2.05 & 2.01 & 3.36 & 3.25 &2.09 & 2.07  & 3.76 &  3.75 & 2.02 & 2.02 \\
       &dim &  13.0 & 12.6  & 14.0  & 13.2  & 14.1 & 13.0 & 14.7 & 14.9 & 10.8 & 10.9 & 11.0 & 11.1 & 11.5 & 11.4 & 12.3 & 12.1 \\
              & & & & &  & & & & & & & & & & & &    \\
       &MSE  &  6.50 &  6.04 & 3.12  & 2.89 & 4.82 & 4.64 & 2.62 &  2.53 & 7.67  & 7.01 & 3.79 & 3.45 & 5.35 & 5.05 &  2.97 & 2.78 \\
 $b_3$&std  & 8.10  & 7.98 &  3.64 & 2.58 &   6.83 &  6.73 & 3.42 & 3.38 & 8.11 & 8.07 & 3.59  & 3.56  & 6.61 & 6.58  &  3.33 & 3.33\\
       &dim  &  10.4 & 9.93 &  11.2 & 10.8 & 11.2 & 12.2 & 11.6 & 13.7 & 12.1 & 11.9 &  13.0& 1.1  & 12.9 &  13.1 & 14.3  & 14.9 \\
              & & & & &  & & & & & & & & & & & &    \\
       & MSE &  8.19 & 7.54 &  4.51  & 4.20 &  9.70 & 9.35 & 4.47 & 4.31& 10.2 & 8.76 &  5.81& 5.00 & 11.1 &  10.2& 5.31 & 4.84 \\
 $b_4$& std   & 9.68  & 9.49 &  6.00 & 5.89 &   11.0 &  10.8 & 5.31 & 5.24 & 9.76& 9.49 & 5.92   & 5.85 & 10.9 & 10.9 & 5.23 & 5.20 \\
       & dim & 11.4  & 11.2 & 12.3 & 12.1 & 12.3 & 12.7 & 13.0 & 13.7 & 10.8 & 10.0 & 11.4 & 11.8 & 12.9 & 11.7 & 15.0 &  14.8 \\
                 & & & & &  & & & & & & & & & & & &    \\
           &MSE &  1.04 & 0.99 & 0.72 & 0.68  & 0.93 & 0.84 & 0.52 & 0.46 & 1.14 & 1.00 &  0.73 & 0.56  &0.92  & 0.87   & 0.53  & 0.49 \\
        $b_5$&std   & 0.79  & 0.78  & 0.45  &  0.43 &   0.93 &  0.96 & 0.40 & 0.39 & 0.80 & 0.78  & 0.44 & 0.42  & 0.94 & 0.93 & 0.41 & 0.40 \\
       &dim & 12.8  & 12.6  & 13.2  & 12.6 & 14.7 & 17.5 & 15.3 & 18.7 & 18.2 & 21.0 & 16.5 & 23.3  & 22.9 & 24.2 & 20.2 & 28.6 \\
              & & & & &  & & & & & & & & & & & &    \\
       &MSE  &  0.90 &  0.83 & 0.48  & 0.43 &  0.85 &  0.80 & 0.37 & 0.35 &1.18 &1.01 & 0.64 & 0.54 & 1.03 & 0.93  & 0.48 & 0.43 \\
 $b_6$&std  &  1.01 & 0.97 & 0.57  & 0.55 & 1.12 & 1.10 & 0.49 & 0.48  &0.01 & 0.99 &  0.56 & 0.55 & 1.10 & 1.09 & 0.49 & 0.48 \\
       &dim  &  9.29 & 9.36 & 10.7  & 9.79 & 10.7 & 9.71 & 12.6  & 9.99 & 12.4 & 11.5 & 13.4 & 13.2 & 17.1 & 16.2 & 18.2 & 18.5 \\
              & & & & &  & & & & & & & & & & & &    \\
       &MSE & 1.32  & 1.25 &   0.76 & 0.69 & 1.16 & 1.10 & 0.62 & 0.57 &  1.63 & 1.49 & 0.86   & 0.80 & 1.34 & 1.24 & 0.70 & 0.65 \\
 $b_7$& std   & 1.51  & 1.46 &  0.74  & 0.73 &  1.42 & 1.39 & 0.75 & 0.73 & 1.50 & 1.48 & 0.73 & 0.71 & 1.35 & 1.34 & 0.73 & 0.73 \\
       & dim & 13.5  & 14.4 & 13.9  & 14.7 & 14.2 & 16.2 & 15.0 & 17.2 & 14.9 & 16.4 & 15.3 & 17.0 & 16.7 &  17.2 &  17.6 & 19.5  \\
  \multicolumn{18}{c}{}  \\
\end{tabular} }
\caption{{\small Results for Gaussian $X$  with 400 repetitions, for $\rho=0.25$ , sample sizes $n=250, 500$ and $J=20,40$, 'Est' the estimator, compared to the oracle ('Or'), for regression functions defined in (\ref{defbi})-(\ref{defbsuite}),  MSE is $1000\times$Relative MISE, std is $1000\times$std, and 'dim' is the mean of the selected dimensions. }}\label{tab1}
\end{table}

\begin{table}
 \hspace{-1cm}\footnotesize{
 \begin{tabular}{cc||cccc|cccc||cccc|cccc}
& & \multicolumn{8}{c||}{Hermite basis} & \multicolumn{8}{c}{Tchebychev basis}\\
&& \multicolumn{4}{c}{$n=250$} & \multicolumn{4}{c||}{$n=500$} & \multicolumn{4}{c}{$n=250$}  & \multicolumn{4}{c}{$n=500$} \\
&& \multicolumn{2}{c}{$J=20$} &  \multicolumn{2}{c|}{$J=40$} &  \multicolumn{2}{c}{$J=20$} &  \multicolumn{2}{c||}{$J=40$} & \multicolumn{2}{c}{$J=20$} &  \multicolumn{2}{c|}{$J=40$} &  \multicolumn{2}{c}{$J=20$} &  \multicolumn{2}{c}{$J=40$} \\
&& Est & Or & Est & Or &  Est & Or & Est & Or & Est & Or & Est & Or & Est & Or & Est & Or \\  \hline
& MSE & 1.31  &  0.89 & 0.74 & 0.52  & 0.60 & 0.45 & 0.33 & 0.25 & 2.85  & 2.10 & 1.72 & 1.29 & 1.38 & 1.02 & 0.82 &  0.62   \\
A & std & 0.79 & 0.53 & 0.46 & 0.26  & 0.33 & 0.24  & 0.18 & 0.13 & 1.61 & 1.05 & 0.85 & 0.69 & 0.73 & 0.55 & 0.35 & 0.28   \\
& dim  & 9.49  & 9.53 & 10.6 & 10.6  & 10.7 & 11.5 & 11.3 & 12.8 & 11.5 & 11.4 &  12.3 & 12.9 & 12.2 & 12.6 & 13.3 & 13.8   \\
& & & & &  & & & & & & & & & & & &    \\
    & MSE & 1.89  & 1.52 & 1.18 & 0.97 & 0.86 & 0.74 & 0.60 & 0.52 & 3.20 & 2.57 & 1.99 & 1.63 & 1.54 & 1.21 & 0.98 &  0.79  \\       
 B & std & 1.50 & 1.39 & 0.97 & 0.89 & 0.79 & 0.75 & 0.55 & 0.52 & 1.86 & 1.62 & 1.20 & 1.08 &  1.00 & 0.89 & 0.64 & 0.57   \\
& dim & 10.2 & 9.56 & 11.2 & 10.8 & 11.1 & 11.9 & 11.8 & 13.0 & 11.9 & 12.3 & 12.8 & 13.0 & 12.8 & 13.2 & 14.4 & 14.4   \\
  \multicolumn{18}{c}{}  \\
\end{tabular} }
\caption{{\small Results for Gaussian $X$  with 400 repetitions, for function $b_4$ in (\ref{defbi}), $A=(\rho=0, \widehat \mu=0)$, and $B=(\rho=0.25, \widehat \mu= \mu_0)$ and sample sizes $n=250, 500$ and $J=20,40$, 'Est' the estimator, compared to the oracle ('Or'),  MSE is $1000\times$Relative MISE, std is $1000\times$std, and 'dim' is the mean of the selected dimensions. }}\label{tab1bis}
\end{table}

\begin{table}
 \hspace{-1cm}\footnotesize{
 \begin{tabular}{cc||cccc|cccc||cccc|cccc}
& & \multicolumn{8}{c||}{Estimated level $\widehat \mu =\bar Y$} & \multicolumn{8}{c}{True level $\mu_0$}\\
&& \multicolumn{4}{c}{$n=62$} & \multicolumn{4}{c||}{$n=125$} & \multicolumn{4}{c}{$n=62$}  & \multicolumn{4}{c}{$n=125$} \\
&& \multicolumn{2}{c}{$J=20$} &  \multicolumn{2}{c|}{$J=40$} &  \multicolumn{2}{c}{$J=20$} &  \multicolumn{2}{c||}{$J=40$} & \multicolumn{2}{c}{$J=20$} &  \multicolumn{2}{c|}{$J=40$} &  \multicolumn{2}{c}{$J=20$} &  \multicolumn{2}{c}{$J=40$} \\
&& Est & Or & Est & Or &  Est & Or & Est & Or & Est & Or & Est & Or & Est & Or & Est & Or \\  \hline
   & MSE & 12.4 & 11.7 & 5.88 & 5.67 & 11.4 & 11.0 & 5.75 & 5.51 & 2.94 & 2.33 & 1.58 & 1.38 & 1.66 & 1.28 & 0.94  & 0.77   \\
  $b_3$ & std & 13.6 & 13.5 & 6.77 & 6.75  & 14.4 & 14.3 & 6.52 & 6.50 & 1.25 & 1.06 & 0.68 & 0.63 & 0.65 & 0.54 & 0.33 & 0.29   \\
   & dim & 10.9 & 10.4 & 10.3 & 11.3 & 13.7 & 11.9 & 13.6 & 14.3 & 11.0 & 10.3 & 10.4 & 11.2 & 13.7 & 11.8 & 13.5 & 14.0    \\
              & & & & &  & & & & & & & & & & & &    \\
   & MSE & 20.6 & 19.4 & 10.8 & 10.4 & 20.6 & 19.8 & 10.1 & 9.83 & 4.08 & 2.89 & 2.18 & 1.73 & 2.34 & 1.61 & 1.27 & 1.00   \\
  $b_4$ & std & 23.1 & 23.0 & 15.1 & 15.1 & 23.8 & 23.7 & 11.7 & 11.6 & 1.75 & 1.31 & 0.78 & 0.64 & 0.90 & 0.63 & 0.45 & 0.35   \\
   & dim & 9.31 & 8.03 & 9.26 & 9.54  & 11.7 & 9.71 & 12.7 & 11.6 & 9.22 & 7.84 & 9.47 & 9.06 & 11.7 & 9.47 & 12.6 & 11.2   \\
              & & & & &  & & & & & & & & & & & &    \\
   & MSE & 2.09 & 1.96 & 1.22 & 1.11 & 1.58 & 1.50 & 0.91 & 0.85 & 0.89 & 0.77 & 0.58 & 0.48 & 0.44 & 0.37 & 0.29 & 0.22   \\
  $b_5$ & std & 1.95 & 1.94 & 1.05 & 1.04 & 1.96 & 1.95 & 0.96 & 0.95 & 0.41 & 0.39 & 0.19 & 0.18 & 0.17 & 0.17 & 0.09 & 0.08    \\
   & dim & 11.9 & 13.3 & 10.8 & 14.1 & 14.8 & 18.3 & 13.9 & 19.8 & 11.9 & 13.3 & 10.8 & 14.3 & 14.9 & 18.3 & 13.9 & 20.1   \\
              & & & & &  & & & & & & & & & & & &    \\
   & MSE & 2.26 & 2.11 & 1.23 & 1.17 & 2.03 & 1.94 & 1.16 & 1.12 & 0.74 & 0.60 & 0.43 & 0.37 & 0.38 & 0.30 & 0.21 & 0.18     \\
  $b_6$ & std & 2.61 & 2.59 & 1.39 & 1.38 & 2.43 & 2.42 & 1.70 & 1.69 & 0.43 & 0.39 & 0.24 & 0.23 & 0.19 & 0.18 & 0.12 & 0.11   \\
   & dim & 10.1 & 8.42 & 9.28 & 9.30 & 12.3 & 9.93 & 12.3 & 11.9 & 9.99 & 8.14 & 9.36 &  9.35 & 12.2 & 9.62 & 12.5 & 11.9   \\
  \multicolumn{18}{c}{}  \\
\end{tabular} }
\caption{{\small Results for non centered uniform $X$ ($2\sqrt{3}{\mathcal U}([0,1])-0.25)$)  with 400 repetitions, for $\rho=0.5$ , sample sizes $n=62, 125$ and $J=20,40$, 'Est' the estimator, compared to the oracle ('Or'), for functions defined in (\ref{defbi})-(\ref{defbsuite}),  MSE is $1000\times$Relative MISE, std is $1000\times$std, and 'dim' is the mean of the selected dimensions. Estimation with Tchebychev basis.}}\label{tab2}
\end{table}

We implement the model selection procedure described in Section \ref{Adaptive}, and get a data-driven choice $\widehat m$ of the dimension $m$ performed among dimensions $m=1, \dots, D_{\max}$, such that the constraint defined for $\Lambda_m$ is fulfilled, that is,
$$\widehat{\mathcal M}_{n,J}=\left\{m\in \{1,\dots, D_{\max}\}, \;  L(m) (\|(\widehat \Psi_m^{(R)})^{-1}\|_{{\rm op}}\vee 1) \leq  \frac{c}{1-\rho} \frac{n \sqrt{J}}{\log(n)\sqrt{1\vee \log(J)}}\right\}, \quad c= 8\times 64 $$
with $L(m)=m(m+1)(2m+1)/(3\pi)$ for the Tchebychev basis and $L(m)=\sqrt{m}$ for the Hermite basis. Since the theoretical value of  $ D_{\max} = nJ$ is too large in practice, we take $D_{\max}=[\sqrt{n}]$ for Hermite basis and $D_{\max}=[2\sqrt{n}]$ for the Tchebychev basis and we check that $D_{\max}$ is never selected, otherwise we increase it.
Concretely, 
$$\widehat m = \arg\min_{m\in \widehat{\mathcal M}_{n,J}}  \left\{-\|\widehat b_m \|_{n,J,R}^2 + {\rm pen}(m)\right\}, \quad {\rm pen}(m)=\kappa \sigma^2 \frac{m}{nJ},$$
with $\kappa=2$ for both bases.

The use of the $n\times n$ matrix $R^{-1}$ involves numerical difficulties which leads us to compute $\widehat\Psi_{m,j}^{(R)}$ with the formula relying of the fact that $R^{-1}=(\alpha_n-\beta_n){\rm Id}_n+ \beta_n {\mathbf 1}{\mathbf 1}^\top$.

The estimators are computed at points $x_k={\tt a}_1+({\tt a}_2-{\tt a}_1)k/K$ for $k=1, \dots, K$,  which are regularly spaced on an interval corresponding to the 1\% and 99\% quantiles of the observations.  In order to keep a common scale for all examples, we consider relative errors (MSE) of any estimator $\widehat b$ of $b$ as the mean over 400 repetitions of the ratios
\begin{equation*}
\frac{\sum_{k=1}^K (\widehat b(x_k)-b(x_k))^2}{\sum_{k=1}^K b^2(x_k)}.
\end{equation*}

Now, as explained in Section \ref{Identif}, there is a level part in the model which must be preliminary estimated, but with rate $1/J$ while the rest of the function can be estimated with a nonparametric rate associated to $nJ$. We experiment the procedure by applying the method to $Y-\widehat \mu$ with $\widehat \mu=(1/nJ) \sum_{i,j} Y_{i,j}$, leading to $\widehat b^\star_{\widehat m}$ and take  $\widehat b_{\widehat m}= \widehat b^\star_{\widehat m} +\widehat \mu$. The results of  experiments done with the functions of (\ref{defbi})-(\ref{defbsuite}) are given in Table \ref{tab1}. They clearly show that the improvements are obtained through increasing $J$, while increasing $n$ does not significantly improve the estimation. For comparison, Table \ref{tab1bis} shows the results for function $b_4$ in two cases: the i.i.d. usual noise case obtained by setting $\rho=0$ and $\widehat \mu=0$, and the same case as in Table \ref{tab1} but with $\widehat \mu$ computed as 
$(1/K)\sum_k b(x_k)$, that is, an approximation of $$\mu_0=\frac 1{{\tt a}_2-{\tt a}_1}\int_{{\tt a}_1}^{{\tt a}_2} b(x)dx,$$ 
the ``good correction", as explained in Section \ref{Identif}. 
Table \ref{tab2} gives the results obtained with the same procedure as in Table \ref{tab1} or with $\widehat \mu$ replaced by  $\mu_0$ for more cases and smaller sample sizes. We also changed the value of $\rho$, and the distribution of $X$ which is no longer symmetric.
We can see several interesting things:
\begin{itemize}
    \item From Table \ref{tab1bis}, we see that the use of $\mu_0$ gives excellent results, rather near of the ones obtained in the independent case when $\rho=0$; the improvement obtained when we compare the use of unknown $\mu_0$ to the use of $\bar Y$ is also really puzzling in Table \ref{tab2}.
    \item Table \ref{tab2} shows that the use of $\mu_0$ decreases the error when $nJ$ increases, whether the increase is due to $n$ or to $J$, and not only for the increase of $J$ like we could see when using $\widehat \mu$ in Tables \ref{tab1} and \ref{tab2}.
    \item The selected dimensions are generally close to the one of the oracles, and therefore very relevantly chosen.
    \item The selected dimensions increase with $nJ$, which is expected, but rather moderately.
\end{itemize}
Globally, all the errors are small (the MSE are multiplied by 1000 in the tables), so that all the results are satisfactory, even for small sample sizes from the point of view of nonparametric estimation. This is why we used smaller values of $n$ in Table \ref{tab2} than in Tables \ref{tab1} and \ref{tab1bis}.
Note that from Table \ref{tab1} to \ref{tab2}, we changed $\rho$ and the distribution of $X$, but these terms do not have an important impact on the results. We do not give the results for all functions for conciseness, as they are all very similar.

\section{Proofs}\label{proofs}
\subsection{Proof of Lemma \ref{linkCov}}
\begin{proof}[Proof of Lemma \ref{linkCov}]
From \cite{comte2025non}, Proposition 2, we have that, for all integers $j$, and vectors ${\mathbf a}\in {\mathbb R}^m$, $$\lim_{n\rightarrow +\infty } {\mathbf a}^\top \widehat \Psi_{m,j}^{(R)} {\mathbf a}= {\mathbf a}^\top \Psi_m {\mathbf a} \quad a.s.$$
Thus, the announced result follows by averaging over $j$ since $\wPsi=(1/J)\sum_{j=1}^J  \widehat \Psi_{m,j}^{(R)}$. 
\end{proof}
\subsection{Proof of Proposition \ref{prop::riskbound-empirical}}
\begin{proof}[Proof of Proposition \ref{prop::riskbound-empirical}]
To consider the quadratic risk of the estimator $\widehat{b}_m$, first define 
\begin{align*}
    \mathbf{a}_m := (\wPsi)^{-1}  \frac{1}{nJ} \wPhi^\top \Rbig b(\Xvec) \in \Rz^{m\times 1}.
\end{align*}
Note that $\mathbf{a}_m = (a_{1},\ldots,a_{m})$ is the orthogonal projection of $b(\Xvec)$ on the space spanned by $\varphi_k(\Xvec)$ with $k\in\{1,\ldots,m\}$ with respect to the scalar product induced by $\frac{1}{nJ}\Rbig$.
Then, $b_m := \sum_{k=1}^m a_{k} \varphi_k$ is the orthogonal projection of $b$ on $S_m$ with respect to the scalar product $\langle \cdot, \cdot \rangle_{n,J,R}$. 
We can use this to obtain 
\begin{align*}
    \Vert b - \widehat{b}_m \Vert_{n,J,R}^2 &= \Vert b - b_m \Vert_{n,J,R}^2 + \Vert \widehat{b}_m - b_m \Vert_{n,J,R}^2\\
    &= \inf_{t\in S_m} \Vert b - t \Vert_{n,J,R}^2 + \Vert \widehat{b}_m - b_m \Vert_{n,J,R}^2.
\end{align*}
For the second summand we insert the definitions and get 
\begin{align*}
    \Vert \widehat{b}_m - b_m \Vert_{n,J,R}^2 &= (\widehat{\mathbf{a}}_m- \mathbf{a}_m)^\top \wPsi (\widehat{\mathbf{a}}_m- \mathbf{a}_m)  \\
    &= (\Yvec - b(\Xvec))^\top \Rbig  \frac{1}{nJ}\wPhi (\wPsi)^{-1} \wPsi (\wPsi)^{-1}  \frac{1}{nJ} \wPhi^\top \Rbig (\Yvec - b(\Xvec))\\
    &= \frac{1}{n^2J^2}(\Yvec - b(\Xvec))^\top \Rbig  \wPhi (\wPsi)^{-1}  \wPhi^\top \Rbig (\Yvec - b(\Xvec)).
\end{align*}
It holds that $ \E[\uvec \uvec^\top] = \operatorname{var}(\uvec) = \sigma^2 \mathbf{R}$.
Finally, we obtain 
\begin{align*}
    \E [ \Vert \widehat{b}_m - b_m \Vert_{n,J,R}^2 ] &= \frac{1}{n^2J^2} \E [\operatorname{Tr}(\uvec^\top \Rbig  \wPhi (\wPsi)^{-1}  \wPhi^\top \Rbig \uvec )]\\
    &= \frac{1}{nJ} \E [\operatorname{Tr}(\Rbig  \wPhi ( \wPhi^\top \Rbig \wPhi)^{-1}  \wPhi^\top \Rbig \uvec \uvec^\top  )]\\
    &= \frac{ \sigma^2 }{nJ} \E [\operatorname{Tr}(\Rbig  \wPhi ( \wPhi^\top \Rbig \wPhi)^{-1}  \wPhi^\top \Rbig \mathbf{R}  )]\\
    &= \frac{ \sigma^2 }{nJ} \operatorname{Tr}({\rm Id}_m) = \frac{ \sigma^2 m}{nJ}
\end{align*}
where we used the independence of $\uvec$ and $\Xvec$ and commuted matrices inside the trace. Combining all the calculations, it follows
\begin{align*}
    \E [ \Vert b - \widehat{b}_m \Vert_{n,J,R}^2 ] &=  \E\left[\inf_{t\in S_m} \Vert b - t \Vert_{n,J,R}^2 \right]+  \frac{ \sigma^2 m}{nJ}\\
    &\leq \inf_{t\in S_m}  \E [ \Vert b - t \Vert_{n,J,R}^2 ] + \frac{ \sigma^2 m}{nJ}\\
    &= \inf_{t\in S_m} \Vert b - t \Vert_{R}^2 + \frac{ \sigma^2 m}{nJ}.
\end{align*}
This concludes the proof. 
\end{proof}

\subsection{Proof of Proposition \ref{prop::new-prop3}}\label{sec:proofprop2}
\begin{proof}[Proof of Proposition \ref{prop::new-prop3}]
  
    Using the first part of Proposition 2 of \cite{comte2025non}, we have that
    \begin{align*}
       \frac 1{\alpha_n} \ePsi = \Psi_m + d_n \Xi_m, \quad d_n=c_n/\alpha_n, 
    \end{align*}
    where $\Xi_m = u_m u_m^\top$ for $u_m := (\E[\varphi_k(X_{1,1})])_{1\leq k \leq m }$. 
   We apply Lemma \ref{lem::sherman-morrison} for $A = \Psi_m$, $v= v_m :=d_n u_m$ and $u = u_m$. As obviously $d_n u_m^\top \Psi_m^{-1} u_m\geq 0$,  it holds that $1+d_nu_m^\top  \Psi_m^{-1} u_m \not=0$ and Lemma \ref{lem::sherman-morrison} yields that $\ePsi$ is invertible with 
    \begin{align*}
        \alpha_n (\ePsi)^{-1} &= \Psi_m^{-1} - \frac{\Psi_m^{-1}  u_m v_m^\top \Psi_m^{-1}}{1+ v_m^\top \Psi_m^{-1}u_m}\\
        &= \Psi_m^{-\frac12} \left( {\rm Id}_m - \frac{d_n\Psi_m^{-\frac12} u_m u_m^\top \Psi_m^{-\frac12}}{1 + d_n u_m^\top \Psi_m^{-1}u_m} \right)\Psi_m^{-\frac12}.
    \end{align*} 
    The matrix $\Psi_m^{-\frac12} u_m u_m^\top \Psi_m^{-\frac12}$ is symmetric, positive definite and has rank 1. So, its eigenvalues are 0 and 
    \begin{align*}
    \operatorname{Tr}(\Psi_m^{-\frac12} u_m u_m^\top \Psi_m^{-\frac12}) = \operatorname{Tr}(u_m^\top \Psi_m^{-1} u_m) = u_m^\top \Psi_m^{-1} u_m.
    \end{align*}
    For the operator norms it follows that
    \begin{align*}
        \alpha_n \Vert (\ePsi)^{-1} \Vert_{\operatorname{op}} &\leq \Vert \Psi_m^{-1} \Vert_{\operatorname{op}}  \left\Vert {\rm Id}_m - \frac{d_n(\ePsi)^{-\frac12} u_m u_m^\top (\ePsi)^{-\frac12}}{1 + d_n u_m^\top \Psi_m^{-1}u_m} \right\Vert_{\operatorname{op}} \\
        &= \Vert \Psi_m^{-1} \Vert_{\operatorname{op}} \max \left( 1, \left\vert 1 - \frac{d_n u_m^\top \Psi_m^{-1} u_m}{1 + d_n u_m^\top \Psi_m^{-1}u_m}\right\vert \right)\\
        &= \Vert \Psi_m^{-1} \Vert_{\operatorname{op}} \max \left( 1, \frac{1}{1 + d_n u_m^\top \Psi_m^{-1}u_m} \right) = \Vert \Psi^{-1}_m \Vert_{\operatorname{op}},
    \end{align*}
    as $d_n u_m^\top \Psi_m^{-1}u_m\geq 0$. 
   We obtain
    \begin{align*}
         \Vert (\ePsi)^{-1} \Vert_{\operatorname{op}} \leq  \frac 1{\alpha_n} \Vert \Psi_m^{-1} \Vert_{\operatorname{op}}.
    \end{align*}
   Combining this with Assumption \ref{ass::theoretical-bound} (ii) and as $1/\alpha_n\leq 2(1-\rho)$ for $n\geq 3$, this finally leads to 
    \begin{align*}
         L(m) (\Vert (\ePsi)^{-1} \Vert_{\operatorname{op}} \vee 1) \leq \frac 1{\alpha_n} L(m) (\Vert \Psi_m^{-1} \Vert_{\operatorname{op}} \vee 1) \leq 2(1-\rho) c^* \frac{n\sqrt{J}}{\log(n)\sqrt{1\vee\log(J)}},
    \end{align*}
    which ends the proof of Proposition \ref{prop::new-prop3}.
\end{proof}

\subsection{Proof of Lemma \ref{TheResult}} 
\begin{proof}[Proof of Lemma \ref{TheResult}]The idea of the proof is to appropriately split the set $\Omega_m$ again. 
For this, let us define the following two sets. For $Z_{n,J}$ defined in \eqref{eq::zJ} we set
\begin{align*}
    \Omega_m^{(1)} := \left\{\sup_{t\in S_m, \operatorname{Var}(t(X_{1,1}))\not= 0} \left\vert\frac{Z_{n,J}(t)}{\operatorname{Var}(t(X_{1,1}))} -1 \right\vert\leq \frac 12 \right\}
\end{align*}
and for $U_{n,J}$ defined in \eqref{eq::uJ}
\begin{equation*}
\Omega_m^{(2)}:=
 \left\{ \sup_{t\in S_m, \|t\| \not= 0} \left\vert\frac{U_{n,J}(t)}{\|t\|^2} \right\vert \leq {\mathfrak c}_r \frac{(m\vee L(m))\sqrt{\log(J)\vee 1}\log(n)}{n\sqrt{J}} \right\}.
\end{equation*}
We show that these sets are, under the given assumptions, indeed suitable choices to control $\mathbb{P}(\Omega_m^c)$ within two steps: 
\begin{itemize} \item First we prove that
\begin{align}\label{eq::step1}
    \Omega_m^{(1)} \cap \Omega_m^{(2)}\subset \Omega_m.
\end{align}
\item Next we bound ${\mathbb P}((\Omega_m^{(1)})^c)$  in Lemma \ref{Tropp} and ${\mathbb P}((\Omega_m^{(2)})^c)$  in Lemma \ref{Ustat}.
\end{itemize}

For the first step, we define for $t\in S_m$
\begin{align*}
     S_{n,J}(t) :=  \alpha_n Z_{n,J}(t)+ c_n {\mathbb E}[t^2(X_{1,1})] + 2c_n {\mathbb E}[t(X_{1,1})] V_{n,J}(t) = \Vert t\Vert_{n,J,R}^2 - (n-1)\beta_n U_{n,J}(t),
\end{align*}
implying, by Lemma \ref{Lemm2}, that 
\begin{equation}\label{TwoTerms}
  \Vert t \Vert_{n,J,R}^2 =  S_{n,J}(t) + (n-1)\beta_n U_{n,J}(t).
\end{equation}
Then the first step follows from the two following Lemmas.
\begin{lem}\label{exLem5} Assume that $n\geq 192/\rho$. Then, on $\Omega_{m}^{(1)}$, for all $ t\in S_m$ it holds $$\left( \frac 12 - \frac 18\right) \|t\|_R^2 \leq S_{n,J}(t) \leq \left( \frac 32 + \frac 18\right) \|t\|_R^2 .$$  
\end{lem}
 
\begin{lem}\label{Omega2Last} Assume Condition (\ref{Stability}) and ${\mathfrak c}_r  c^\star \leq 1/8$, then, on $\Omega_m^{(2)}$, for all $ t \in S_m$ it holds 
$$(n-1)\beta_n |U_{n,J}(t)|\leq \frac 18 \|t\|_R^2. $$
\end{lem}

Indeed, combining these two results with equality (\ref{TwoTerms}), one concludes that on $\Omega_m^{(1)}\cap \Omega_m^{(2)}$ it holds for $t\in S_m$ with $\Vert t\Vert^2\not=0$ that
\begin{align*}
    \frac{1}{4} \Vert t\Vert^2_R = \left(\frac12 -\frac18 -\frac18\right)\Vert t\Vert^2_R \leq \Vert t\Vert_{,n,R,J}^2 \leq \left(\frac32  + \frac18 +\frac18\right)\Vert t\Vert^2_R = \frac74 \Vert t\Vert^2_R.
\end{align*}

Next the two following Lemmas give the result of the second step. 
\begin{lem}\label{Tropp} Let Assumption \ref{ass::theoretical-bound} be satisfied. For $p\geq 1$, assume that for the constant $c^\star$ of stability condition \eqref{Stability} it holds $c^\star\leq  {\mathbf c}_p^\star/3$ for ${\mathbf c}_p^\star$ defined in \eqref{eq::constants-def}.  Then, it holds under Assumption \ref{ass::theoretical-bound}, for $n\geq3$ and $J \leq n(n-1)/\log^2(n)$ that $${\mathbb P}((\Omega_m^{(1)})^c)\leq 2(nJ)^{-p} .$$ 
\end{lem}
\begin{lem}\label{Ustat}    
Under Assumption \ref{ass::theoretical-bound} it holds for $n\geq 3$, $J \leq n(n-1)/\log^2(n)$ and $r\geq 1$ that
$${\mathbb P}((\Omega_m^{(2)})^c)\leq  5.44 \,  \frac{1}{(nJ)^{r-2\min(1/s,1)}}.$$  
\end{lem}
Finally, gathering \eqref{eq::step1}, Lemma \ref{Tropp} and Lemma \ref{Ustat} gives the result, i.e.
 \begin{align*}
     \mathbb{P}(\Omega_m^c)\leq \mathbb{P}((\Omega_m^{(1)})^c) + \mathbb{P}((\Omega_m^{(2)})^c) \leq 12 (nJ)^{-\min( p , r-2/s)}.
 \end{align*}
The proofs of Lemmas \ref{exLem5}-\ref{Ustat} are given in Section \ref{ProofLemmas} below.
\end{proof}

\subsection{Proofs of Lemmas \ref{exLem5}-\ref{Ustat}}\label{ProofLemmas}
We use the convention $\|\Psi_m^{-1}\|_{\rm op}=+\infty$ if $\Psi_m$ is not invertible. 
\begin{proof}[Proof of Lemma \ref{exLem5}] Recall that we consider 
$$S_{n,J}(t) = \alpha_n Z_{n,J}(t)+ c_n {\mathbb E}[t^2(X_{1,1})] + 2c_n {\mathbb E}[t(X_{1,1})] V_{n,J}(t)= \Vert t\Vert_{n,J,R}^2 - (n-1)\beta_n U_{n,J}(t).$$
On $\Omega_m^{(1)}$,  we have for all $ t\in S_m$ that
$$ \frac 12 {\rm Var}(t(X_{1,1}) \leq Z_{n,J}(t) \leq \frac 32 {\rm Var}(t(X_{1,1})).$$
Therefore, using  $$\frac 12c_n({\mathbb E}[t(X_{1,1})])^2\leq c_n({\mathbb E}[t(X_{1,1})])^2\leq \frac 32c_n({\mathbb E}[t(X_{1,1})])^2$$ and $\alpha_n{\rm Var}(t(X_{1,1}))+ c_n({\mathbb E}[t(X_{1,1})])^2=\|t\|_R^2$, we get 
\begin{equation}\label{step1} \frac 12 \|t\|_R^2 \leq \alpha_n Z_{n,J}(t) + c_n({\mathbb E}[t(X_{1,1})])^2\leq \frac 32 \|t\|_R^2.
\end{equation}
Now, it holds for all $j=1,\ldots, J$ that $V_{n,j}^2(t)\leq Z_{n,j}(t)$. Following the same proof as in Lemma 5 of \cite{comte2025non}, this implies  that on $\Omega_m^{(1)}$,  for $n\geq 192/\rho$, it holds
$$2c_n |{\mathbb E}[t(X_{1,1})] V_{n,j}(t)|\leq \frac{\sqrt{3}}{n\rho} \|t\|_R^2\leq \frac 18 \, \|t\|_R^2.$$
Thus, averaging over $j$ yields 
$$2c_n |{\mathbb E}[t(X_{1,1})] V_{n,J}(t)|\leq \frac 18 \, \|t\|_R^2.$$
As a consequence, this inequality combined with (\ref{step1}) imply
$$(\frac 12 -\frac 18) \|t\|_R^2 \leq S_{n,J}(t) \leq (\frac 32 + \frac 18)\|t\|_R^2, $$
which ends the proof.
\end{proof}

\begin{proof}[Proof of Lemma \ref{Omega2Last}]
On $\Omega_m^{(2)}$, for any $t\in S_m$, 
$$|U_{n,J}(t)|\leq {\mathfrak c}_r \|\Psi_m^{-1}\|_{{\rm op}}  \frac{(m\vee L(m))\sqrt{\log(J)}\log(n)}{n\sqrt{J}}{\rm Var}(t(X_{1,1})).  $$
Under Condition (\ref{Stability}), we get 
$$|U_{n,J}(t)|\leq {\mathfrak c}_r  c^* {\rm Var}(t(X_{1,1}).$$
So, as $(n-1)\beta_n/\alpha_n\leq 1$ and $\alpha_n {\rm Var}(t(X_1))\leq \|t\|_R^2$, we get
$$(n-1)\beta_n |U_{n,J}(t)|\leq {\mathfrak c}_r  c^\star ((n-1)\beta_n/\alpha_n) \|t\|_R^2\leq  {\mathfrak c}_r  c^* \|t\|_R^2\leq \frac 18 \|t\|_R^2,$$
as ${\mathfrak c}_r  c^*\leq 1/8$, which ends the proof of Lemma \ref{Omega2Last}. 
\end{proof}

\begin{proof}[Lemma \ref{Tropp}]
The result is analogous to Lemma 3 in \cite{comte2025non}. 
First, note that  due to $J \leq n(n-1)/\log^2(n)$ it also holds that
 \begin{align}
     \frac{\log(nJ)}{\log(n)\sqrt{\log(J)\vee 1}} \leq \frac{3\log(n)}{\log(n)\sqrt{\log(J)\vee 1}}\leq 3 \log(n).\label{eq::step-n-j}
 \end{align}
 Combining this with the stability condition \eqref{Stability} it follows that
 \begin{align}
     L(m)\|\Psi_m^{-1}\|_{{\rm op}} &\leq c^\star  \frac{n\sqrt{J}}{\log(n)\sqrt{1\vee\log(J)}} \leq \frac{{\mathbf c}_p^\star}{3} \frac{nJ}{\log(nJ)} \frac{n\sqrt{J}\log(nJ)}{nJ\log(n)\sqrt{1\vee\log(J)}} \nonumber\\
     &\leq {\mathbf c}_p^\star \frac{nJ}{\log(nJ)},\label{TroppCond}
 \end{align}

where we recall $\Psi_m=\left( {\rm cov}(\varphi_j(X_{1,1}), \varphi_k(X_{1,1})\right)_{1\leq j, k\leq m}$.
Here, $\widehat{\Psi}_m$ is defined as
\begin{align*}
    \widehat{\Psi}_m:= \left( \frac{1}{nJ}\sum_{j=1}^J\sum_{i=1}^n \left((\varphi_{k_1}(X_{i,j}) - \E[\varphi_{k_1}(X_{1,1})]) (\varphi_{k_2}(X_{i,j}) - \E[\varphi_{k_2}(X_{1,1})]) \right) \right)_{1\leq k_1,k_2 \leq m}
\end{align*}
and we have a decomposition  of the matrix of interest as a sum over $i,j$ of
\begin{align*}
    K_m(X_{i,j}) =  \Psi_m^{-1/2} \left( \left((\varphi_{k_1}(X_{i,j}) - \E[\varphi_{k_1}(X_{1,1})]) (\varphi_{k_2}(X_{i,j}) - \E[\varphi_{k_2}(X_{1,1})]) \right) \right)_{1\leq k_1,k_2 \leq m}  \Psi_m^{-1/2}.
\end{align*}
So, the proof follows from Tropp Chernov inequality, see Lemma \ref{lem::tropp-chernoff} in the appendix from \cite{tropp2012user}, as used in the proof of Lemma 3 in \cite{comte2025non}, but involving $nJ$ independent and identically distributed random matrices instead of $n$, under (\ref{TroppCond}) implied by (\ref{Stability}). 
\end{proof}

\begin{proof}[Proof of Lemma \ref{Ustat}]
    This is result the generalization of Lemma 4 from \cite{comte2025non}.
    First note that, with Cauchy Schwarz Inequality, for $t=\sum_{k=1}^m a_k\varphi_k$ and $\|t\|^2=\sum_{k=1}^m a_k^2=1$, we get
\begin{eqnarray*} 
U_{n,J}(t)&= & \sum_{1\leq k, \ell \leq m} a_k a_\ell \left(\frac 1J \sum_{j=1}^J U_{n,j}(\varphi_k, \varphi_\ell) \right) 
\leq \left(\sum_{1\leq k,\ell \leq m}\left( \frac 1J\sum_{j=1}^J U_{n,j}(\varphi_k, \varphi_\ell)\right)^2\right)^{1/2}.
\end{eqnarray*}
Note that using again \eqref{eq::step-n-j} we have that
\begin{align*}
    {\mathfrak c}_r \frac{L(m)\sqrt{\log(J)\vee 1}\log(n)}{n\sqrt{J}} &\geq 3\sqrt{2}r^2 c_0 (1+\theta^2) 8\sqrt{6}  \frac{\sqrt{\log(J)\vee 1}\log(n)\sqrt{n(n-1)J}}{n\sqrt{J}\log(nJ)}  \frac{L(m)\log(nJ)}{\sqrt{n(n-1)J}}\\
    &\geq \sqrt{2}r^2 c_0 (1+\theta^2) 8\sqrt{6}  \frac{\sqrt{n(n-1)J}}{n\sqrt{J}}  \frac{L(m)\log(nJ)}{\sqrt{n(n-1)J}}\\
    &\geq r^2 c_0 (1+\theta^2) 8\sqrt{6} \frac{L(m)\log(nJ)}{\sqrt{n(n-1)J}}.
\end{align*}
Consequently, we can write 
\begin{eqnarray}
{\mathbb P}((\Omega_m^{(2)})^c)
& \leq  & {\mathbb P}\left( \sup_{t\in S_m, \|t\| \not= 0} \left\vert\frac{U_{n,J}(t)}{\|t\|^2} \right\vert > 8\sqrt{6} r^2 c_0  \log(nJ) \frac{L(m)(1+ \theta^2)}{\sqrt{Jn(n-1)}} \right) \nonumber \\
& \leq & {\mathbb P}\left( \left(\sum_{1\leq p,q  \leq m} U_{n,J}^2(\varphi_p, \varphi_q)\right)^{1/2}  > 8\sqrt{6} c_0 r^2 \log(nJ) \frac{L(m)(1+\theta^2)}{\sqrt{Jn(n-1)}} \right)\nonumber\\
& =  & {\mathbb P}\left( \sum_{1\leq p,q  \leq m} U_{n,J}^2(\varphi_p, \varphi_q)  > 2\cdot 192 c_0^2 r^4 \log^2(nJ) \frac{L^2(m)(1+\theta^2)^2}{Jn(n-1)} \right).\label{eq::proof-omega2}
\end{eqnarray}
The idea is to apply the U-statistic concentration inequality in  \cite{GineNickl2016Book} (Theorem 3.4.8 p. 183). For this,  we consider for fixed $p,q\in\{1,\ldots,m\}$ 
$$U_{J,n}(\varphi_p,\varphi_q) = \frac 1J\sum_{j=1}^J U_{n,j}(\varphi_p, \varphi_q)=\frac 1{n(n-1)} \sum_{i=2}^n\sum_{\ell=1}^{i-1} g_{p,q}(X_{i,\bullet},X_{\ell,\bullet}),$$
with $$g_{p,q}(X_{i,\bullet},X_{\ell,\bullet}):= \frac 1J \sum_{j=1}^J [\bar \varphi_p(X_{i,j})\bar \varphi_q(X_{\ell, j}) + \bar\varphi_p(X_{\ell, j}) \bar\varphi_q(X_{i,j})] \quad \mbox{ and }  \quad \bar\varphi(x)=\varphi(x)-\int \varphi(y) \, f(y) dy.$$ 
Under Assumption \ref{ass::theoretical-bound} (i), the basis functions $\varphi_p$ are bounded by $\theta$. Consequently, we have that $U_{J,n}(\varphi_p, \varphi_q)$ is a canonical U-statistic with symmetric and bounded kernel $g_{p,q}$ with i.i.d. vector observations $(X_{i, \bullet})_{1\leq i\leq n}$. We obtain the following quantities for \cite{GineNickl2016Book} (Theorem 3.4.8 p. 183):
\begin{align*}
    A&=8\theta^2, \quad C=D=2\sqrt{\frac{n(n-1)}J} \sqrt{\int \varphi_p^2(x) f(x) dx \; \int \varphi_q^2(x) \,f(x) dx}, \\
    B^2 &=4\theta^2\frac{(n-1)}J \left( \int \varphi_p^2(x) \, f(x) dx + \int \varphi_q^2(x) \,f(x) dx\right).
\end{align*}
For $c_0$ denoting the maximum of all the numerical constants in the deviation inequality, i.e. $c_0 = 18.6$ by choosing $\varepsilon = 2$, 
and setting
$$\delta_{p,q}:= \frac{c_0}{n(n-1)} \left[ C\sqrt{r\log(nJ)} + D(r\log(nJ)) + B(r\log(nJ))^{3/2} + A (r\log(nJ))^2\right],$$ the deviation inequality yields for $u= r\log (nJ)$, $r\geq 1$,  that
\begin{align}
\mathbb{P}(\vert U_{J,n}(\varphi_p, \varphi_q)\vert \geq \delta_{p,q}) \leq 2e^{1- r\log(nJ)} \leq 5.44(nJ)^{-r}.\label{eq::deviation-ineq}
\end{align}
Since we assume $n\geq 3$ and $J\geq 1$ we have also that $nJ\geq e$ and, consequently, 
$$\delta_{p,q}\leq c_0r^2 \log(nJ) \left[ 4\frac{\sqrt{\int \varphi_p^2 \, f \; \int \varphi_q^2 \,f}}{\sqrt{n(n-1)J}} + 2\theta \sqrt{\log(nJ)}\frac{\sqrt{\int \varphi_p^2 \, f} + \sqrt{\int \varphi_q^2 \,f}}{\sqrt{J}n \sqrt{n-1}} + \frac{8\theta^2\log(nJ)}{n(n-1)}\right].$$
For $J\leq n(n-1)/\log^2(n)$, then $\log^2(nJ)/[n(n-1)]\leq 9/J$ and $\sqrt{\log(nJ)/n} \leq \sqrt{3}\leq 2$, thus
\begin{align*}
    \delta_{p,q}\leq \frac{c_0 r^2 \log(nJ)}{\sqrt{n(n-1)J}} \left[ 4\sqrt{\int \varphi_p^2 \, f \; \int \varphi_q^2 \,f} + 4\theta \left(\sqrt{\int \varphi_p^2 \, f} + \sqrt{\int \varphi_q^2 \,f}\right) + 8\theta^2 \right]^2.
\end{align*}
Using then that $\sum_{1\leq p,q\leq m}(\sqrt{\int \varphi_p^2 \, f} + \sqrt{\int \varphi_q^2 \,f})^2 \leq 2m L(m) + 2mL(m)= 4L(m)m \leq 4(m\vee L(m))^2$ one obtains
\begin{align*}
    \sum_{1\leq p,q\leq m} \delta_{p,q}^2 &\leq \frac{c_0^2 r^4 \log^2(nJ)}{n(n-1)J} \left(48 L^2(m) + 192 \theta^2 mL(m) + 192 m^2\theta^4  \right) \\
    &\leq \frac{c_0^2 r^4 \log^2(nJ)}{n(n-1)J} 2\cdot 192 (m\vee L(m))^2(1+\theta^2)^2.
\end{align*}
Inserting the last equation and \eqref{eq::deviation-ineq} into \eqref{eq::proof-omega2}, we obtain
\begin{eqnarray*}
{\mathbb P}((\Omega_m^{(2)})^c)& \leq & {\mathbb P}\left( \sum_{1\leq p,q  \leq m} U_{n,J}^2(\varphi_p, \varphi_q)  > \sum_{1\leq p,q\leq m} \delta_{p,q}^2\right)\\
&\leq &  \sum_{1\leq p,q  \leq m} {\mathbb P}\left( \vert U_{n,J}^2(\varphi_p, \varphi_q) \vert  > \delta_{p,q}\right)\\
&\leq & 5.44 \,  \frac{m^2}{(nJ)^r}\leq  5.44 \,  
\frac{1}{(nJ)^{r-2\min(1/s,1)}},
\end{eqnarray*}
where in the last inequality we used that $m^s \leq n\sqrt{J}\leq nJ$ if $s\geq 1$, and $m\leq nJ$ if $s<1$, which yields $m\leq (nJ)^{\frac 1s \vee 1}$. 
\end{proof}

\subsection{Proofs of Lemmas \ref{OmegaXi} - \ref{lem::bound-emp-process}}\label{sec::proofs-lemmas-adaptive}
\begin{proof}[Proof of Lemma \ref{OmegaXi}]
We acknowledge \cite{huang:hal-04901917} for a decisive improvement of the result obtained in this lemma. Consider $\omega\in\Omega_{n,J}$ and
\begin{displaymath}
\widehat G_{m}(\omega) :=
(\Psi_{m}^{(R)})^{-\frac{1}{2}}\widehat\Psi_{m}^{(R)}(\omega)(\Psi_{m}^{(R)})^{-\frac{1}{2}} \textrm{ $;$ }\forall  m\in\mathcal M_{n,J}^{+}.
\end{displaymath}
For any $ m\in\mathcal M_{n,J}^{+}$, since $\omega\in\Omega_{m}$,
\begin{displaymath}
{\rm Sp}(\widehat G_{m}(\omega))
\subset\left[\frac{1}{4},\frac{7}{4}\right],
\quad\textrm{and then}\quad
{\rm Sp}(\widehat G_{m}^{-1}(\omega))
\subset\left[\frac{4}{7},4\right].
\end{displaymath}
Moreover,
\begin{displaymath}
(\widehat\Psi_{m}^{(R)})^{-1}(\omega) =
(\Psi_{m}^{(R)})^{-\frac{1}{2}}\widehat G_{m}^{-1}(\omega)(\Psi_{m}^{(R)})^{-\frac{1}{2}}.
\end{displaymath}
So, thanks to a well-known property of the Loewner order,
\begin{displaymath}
\frac{4}{7}\mathbf x^\top (\Psi_{m}^{(R)})^{-1}\mathbf x
\leqslant\mathbf x^\top (\widehat\Psi_{ m}^{(R)})^{-1}(\omega)\mathbf x\leqslant
4\mathbf x^\top (\Psi_{m}^{(R)})^{-1}\mathbf x
\textrm{ $;$ }\forall\mathbf x\in\mathbb R^{m},
\end{displaymath}
leading to 
\begin{displaymath}
\|(\Psi_{m}^{(R)})^{-1}\|_{\rm op}\leqslant
\frac{7}{4}\|(\widehat\Psi_{ m}^{(R)})^{-1}(\omega)\|_{\rm op}
\quad {\rm and}\quad
\|(\widehat\Psi_{ m}^{(R)})^{-1}(\omega)\|_{\rm op}\leqslant
4\|(\Psi_{m}^{(R)})^{-1}\|_{\rm op}
\end{displaymath}
almost surely. Therefore, $\omega\in\Xi_{n,J}$. \end{proof}

\begin{proof}[Proof of Lemma \ref{lem::bound-on-complement}]
    First we further decompose the risk in two terms, more precisely, 
   $${\mathbb E}[ \|\widehat b_{\widehat m}-b\|_{n,J,R}^2 {\mathbf 1}_{\Omega_{n,J}^c}] \leq 2 
{\mathbb E}[ \|\widehat b_{\widehat m}\|_{n,J,R}^2 {\mathbf 1}_{\Omega_{n,J}^c}] + 2{\mathbb E} [\|b\|_{n,J,R}^2 {\mathbf 1}_{\Omega_{n,J}^c}]:= 2{\mathbb S}_1 + 2{\mathbb S}_2. $$
We first consider ${\mathbb S}_2$. By Lemma \ref{lem::norm-properties}, it holds that $\|b\|_{n,J,R}^2=\frac 1{nJ} b({\mathbf X}_{\bullet,\bullet})^\top {\mathbf R}^{-1}b({\mathbf X}_{\bullet,\bullet}) \leq \frac 1{1-\rho}\|b\|_{n,J}^2$.
In addition, we get
$${\mathbb E}( \|b\|_{n,J,R}^4)\leq \frac 1{(1-\rho)^2} {\mathbb E}\left(\frac 1{nJ} \sum_{j=1}^J \sum_{i=1}^n b^4(X_{i,j})\right)= \frac{{\mathbb E}(b^4(X_{1,1}))}{(1-\rho)^2}.$$
{Note that since by Assumption \ref{ass::theoretical-bound} (i) it holds  $L(m)=m^s$, and consequently $\operatorname{card}(\mathcal{M}_{n,J}^+) \leq (7c^{\star\star} n\sqrt{J})^{1/s}$. } Applying Lemma  \ref{TheResult} with $p=  2 +  1/s $ and $r=2 + 3/s $, we obtain 
\begin{align}
    {\mathbb P}(\Omega_{n,J}^c) \leq 12\operatorname{card}(\mathcal{M}_{n,J}^+) (nJ)^{-\min(p,r-2/s)} \leq 12 (7 c^{\star\star})^{1/s} (nJ)^{- 2}.\label{eq::bound-omega-comp}
\end{align}
It follows by Cauchy-Schwarz inequality that 
\begin{align*}
    2{\mathbb S}_2  \leq 2 \sqrt{12}(7 c^{\star\star})^{1/(2s)} \frac{{\mathbb E}(b^4(X_{1,1}))}{(1-\rho)^2} \frac{1}{nJ}.
\end{align*}
Let us now consider ${\mathbb S}_1$.
First note that 
$$\sup_{{\mathbf x}\in {\mathbb R}^{nJ}, \|{\mathbf x}\|_{{\mathbb R}^{nJ}}=1} {\mathbf x}^\top \widehat\Phi_{\widehat m}\widehat\Phi_{\widehat m}^\top{\mathbf x} = \sup_{{\mathbf x}\in {\mathbb R}^{nJ}, \|{\mathbf x}\|_{{\mathbb R}^{nJ}}=1}\sum_{k=1}^{\widehat m} \left(\sum_{j=1}^J \sum_{i=1}^n x_{i,j}\varphi_k(X_{i,j})\right)^2 \leq (nJ) L(\widehat m), $$ 
where $\Vert \cdot\Vert_{{\mathbb R}^{nJ}}$ denotes the standard Euclidean norm on $\mathbb{R}^{nJ}$.
Further, we have that  $\widehat m\in \widehat{{\mathcal M}}_{n,J}$ and $L(\widehat m)\|(\widehat \Psi_{\widehat m}^{(R))})^{-1} \|_{{\rm op}} \leq 4c^{\star\star} n\sqrt{J}/(\log(n)\sqrt{\log(J)\vee 1})$. With  this we get 
\begin{align*}
{\mathbb S}_1 &= \frac 1{(nJ)^2} {\mathbb E}\left[ {\mathbf Y}_{\bullet,\bullet}^\top {\mathbf R}^{-1} \widehat\Phi_{\widehat m} (\widehat \Psi_{\widehat m}^{(R)})^{-1} \widehat \Phi_{\widehat m}^\top {\mathbf R}^{-1} {\mathbf Y}_{\bullet,\bullet}  {\mathbf 1}_{\Omega_{n,J}^c} \right] \\
&\leq 4c^{\star\star}\frac{1}{nJ} \frac{n\sqrt{J}}{\log(n)\sqrt{\log(J)\vee 1}} {\mathbb E}[  \Vert {\mathbf R}^{-1} {\mathbf Y}_{\bullet,\bullet} \Vert^2_{{\mathbb R}^{nJ}}{\mathbf 1}_{\Omega_{n,J}^c}].
\end{align*}
Now as $\|{\mathbf R}^{-1}\|_{{\rm op}}\leq 1/(1-\rho)$ applying Cauchy-Schwarz inequality it holds that
\begin{align*}
     {\mathbb E}[  \Vert {\mathbf R}^{-1} {\mathbf Y}_{\bullet,\bullet} \Vert^2_{{\mathbb R}^{nJ}}{\mathbf 1}_{\Omega_{n,J}^c}] &\leq \frac{1}{(1-\rho)^2} \sum_{j=1}^J\sum_{i=1}^n \E[Y_{i,j}^2 {\mathbf 1}_{\Omega_{n,J}^c}]\\
     &\leq \frac{nJ}{(1-\rho)^2} \sqrt{{\mathbb E}[Y_{1,1}^4]} \sqrt{{\mathbb P}(\Omega_{n,J}^c)}
 \end{align*}
Using again Lemma \ref{TheResult}  with $p=4 +  1/s $ and $r=4 + 3/s $ analogously to \eqref{eq::bound-omega-comp}, we obtain  the upper bound ${\mathbb P}(\Omega_{n,J}^c)\leq 12(7c^{\star\star})^{1/s}(nJ)^{-4}$ and, consequently, 
\begin{align*}
    {\mathbb S}_1 &\leq  \frac{4c^{\star\star}}{(1-\rho)^2} \frac{n\sqrt{J}}{\log(n)\sqrt{\log(J)\vee 1}} \sqrt{{\mathbb E}[Y_{1,1}^4]} \sqrt{{\mathbb P}(\Omega_{n,J}^c)}\\
    &\leq 12(7c^{\star\star})^{1/s}\frac{4c^{\star\star}}{(1-\rho)^2} \sqrt{{\mathbb E}[Y_{1,1}^4]} \frac{1}{nJ}.
\end{align*}
This finally yields the result, that is,
\begin{align*}
    {\mathbb E}[ \|\widehat b_{\widehat m}-b\|_{n,J,R}^2 {\mathbf 1}_{\Omega_{n,J}^c}] \leq \left( 24 (7c^{\star\star})^{1/s}\frac{4c^{\star\star}}{(1-\rho)^2} \sqrt{{\mathbb E}[Y_{1,1}^4]}+ 2 \sqrt{12}(7 c^{\star\star})^{1/(2s)} \frac{{\mathbb E}(b^4(X_{1,1}))}{(1-\rho)^2}\right)  \frac{1}{nJ}.
\end{align*}

\end{proof}

\begin{proof}[Proof of Lemma \ref{lem::bound-emp-process}]
    First, we define the $n\times (n+1)$ matrix $M$ such that ${\mathbf u}_{\bullet, j}=M  \varepsilon_{\bullet, j}$, where $\varepsilon_{\bullet, j}=(\varepsilon_{0,j}, \varepsilon_{1,j}, \dots, \varepsilon_{n,j})^\top$, more precisely, 
$$M:=\left(\begin{array}{cccccc}
\sqrt{\rho} & \sqrt{1-\rho} & 0 & 0 & \dots & 0 \\ 
 \sqrt{\rho}  &  0 & \sqrt{1-\rho} & 0 & \dots & 0 \\ 
 \vdots  &  \ddots &   &  \ddots & \ddots &  \vdots\\ 
 \vdots &  & \ddots && \ddots & \vdots \\ 
 \sqrt{\rho} & 0 &  \dots & 0  & \sqrt{1-\rho} & 0 \\
\sqrt{\rho} & 0 & \dots  & 0 & 0&  \sqrt{1-\rho} 
\end{array}\right).$$
The first column is a $n$-dimensional vector with coordinates all equal to $\sqrt{\rho}$, concatenated with $\sqrt{1-\rho} \,  {\rm Id}_{n}$. Moreover, we set ${\mathbf M}={\rm diag}(M, \dots, M)\in {\mathbb R}^{nJ\times (n+1)J}$.  Then, we can write 
$ \uvec= {\mathbf M}\varepsilon_{\bullet,\bullet}$, and 
$$\nu_{n,J}(t)=\frac 1{nJ} \uvec^\top \,{\mathbf R}^{-1}t(\Xvec) = \frac 1{nJ} \varepsilon_{\bullet,\bullet}^\top {\mathbf M}^\top {\mathbf R}^{-1}t(\Xvec) =\frac 1{nJ}  \sum_{i=0}^n\sum_{j=1}^J  \varepsilon_{i,j} \underbrace{[M^\top R^{-1} t({\mathbf X}_{\bullet,j})]_{i+1}}_{=\psi_{t,i,j}({\mathbf X})}.$$
Conditionally to  $\Xvec={\mathbf x}_{\bullet,\bullet}$, the term
$$\frac 1{nJ} \sum_{j=1}^J \sum_{i=0}^n \varepsilon_{i,j} \underbrace{[M^\top R^{-1} t({\mathbf x}_{\bullet,j})]_{i+1}}_{=\psi_{t,i,j}({\mathbf x}_{\bullet,j})}$$
is linear with respect to $t$, centered and a combination of independent variables.
Thus, we intend to apply the Talagrand Inequality. But for this, we need to be in a bounded context. As the $\varepsilon_{i,j}$'s are not supposed to be bounded, we rely on Assumption \ref{subExp} and  we split $\varepsilon_{i,j} =\varepsilon_{i,j}^{(1)}+\varepsilon_{i,j}^{(2)}$ and accordingly $\nu_{n,J}=\nu_{n,J,1} + \nu_{n,J,2}$, with $\varepsilon_{i,j}^{(2)}  :=\varepsilon_{i,j}-\varepsilon_{i,j}^{(1)}$ and $$\varepsilon_{i,j}^{(1)}:=\varepsilon_{i,j}{\mathbf 1}_{|\varepsilon_{i,j}|\leq K} - {\mathbb E}(\varepsilon_{i,j}{\mathbf 1}_{|\varepsilon_{i,j}|\leq K}). $$
Accordingly, we denote $\varepsilon_{\bullet,\bullet} = \varepsilon_{\bullet,\bullet}^{(1)} + \varepsilon_{\bullet,\bullet}^{(2)}$ and $\uvec^{(i)} = {\mathbf M}\varepsilon_{\bullet,\bullet}^{(i)}$ for $i=1,2$.
Then, noting that it holds $S_m+S_{\widehat m} =S_{m\vee \widehat m}$ since the spaces are nested and $m\vee \widehat m\in {\mathcal M}_{n,J}^+$ on $\Omega_{n,J}$ due to Lemma \ref{OmegaXi}, we have 
\begin{eqnarray*} && {\mathbb E}\left[ \left(\sup_{t\in S_m+S_{\widehat m}, \|t\|_{n,J,R}=1} \nu_{n,J}^2(t) - p(m,\widehat m) \right)_+\!\!{\mathbf 1}_{\Omega_{n,J}}\right] \\ &\leq & 2  \sum_{m'\in {\mathcal M}_{n,J}^+} {\mathbb E}\left[\left(\sup_{t\in S_{m\vee m'}, \|t\|_{n,J,R}=1} \nu_{n,J,1}^2(t) - \frac{p(m,m')}2 \right)_+\right] + 2 
{\mathbb E} \left[\sup_{t\in S_m+S_{\widehat m}, \|t\|_{n,J,R}=1} \nu_{n,J,2}^2(t) \right]. 
\end{eqnarray*}
Denoting by ${\mathbb E}_{{\mathbf X}}$ the conditional expectation with respect to $\Xvec$, we define for $m'\in\mathcal{M}_{n,J}^+$
\begin{align*}
{\mathbb E}_1(m'):={\mathbb E}_{{\mathbf X}}\left[\left(\sup_{t\in S_m', \|t\|_{n,J,R}=1} \nu_{n,J,1}^2(t) - 4\sigma^2 \frac{m'}{nJ} \right)_+\right]
\end{align*}
and
\begin{align*}
    {\mathbb E}_2:={\mathbb E}\left[\sup_{t\in S_m+S_{\widehat m},, \; \|t\|_{n,J,R}=1} 
\nu_{n,J,2}^2(t) \right].
\end{align*}
Then, we have that
\begin{align*}
     {\mathbb E}\left[ \left(\sup_{t\in S_m+S_{\widehat m}, \|t\|_{n,J,R}=1} \nu_{n,J}^2(t) - p(m,\widehat m) \right)_+
     \right] 
     \leq 2  \sum_{m'\in {\mathcal M}_{n,J}^+} \E[{\mathbb E}_1(m'\vee m)] + 2 {\mathbb E}_2.
\end{align*}
    Let us first bound $ {\mathbb E}_2$. First, we see that
    $${\mathbb E}_2 = \frac 1{(nJ)^2} {\mathbb E}\left[\sup_{t\in S_m+S_{\widehat m}, \;  \|t\|_{n,J,R}=1} ((\varepsilon_{\bullet,\bullet}^{(2)})^\top {\mathbf M}^\top {\mathbf R}^{-1}t({\mathbf X}_{\bullet,\bullet}))^2\right]. $$
    As for $\|t\|_{n,J,R}=1$ it holds 
    \begin{align*}
((\varepsilon_{\bullet,\bullet}^{(2)})^\top {\mathbf M}^\top {\mathbf R}^{-1}t({\mathbf X}_{\bullet,\bullet} ))^2 &\leq \| {\mathbf R}^{-1/2} {\mathbf M} \varepsilon_{\bullet,\bullet}^{(2)}\|^2_{{\mathbb R}^{nJ}} \|{\mathbf R}^{-1/2}t({\mathbf X}_{\bullet,\bullet})\|^2_{{\mathbb R}^{nJ}}\\
&=(nJ) \|{\mathbf R}^{-1/2}{\mathbf u}_{\bullet,\bullet}^{(2)}\|_{{\mathbb R}^{nJ}}^2\leq \frac{nJ}{1-\rho} \|{\mathbf u}_{\bullet,\bullet}^{(2)}\|^2_{{\mathbb R}^{nJ}}
    \end{align*}
we get, since $\varepsilon_{i,j}$ are i.i.d. for all $i=1,\ldots,n$ and $j=0,\ldots,J$, that
$${\mathbb E}_2\leq \frac 1{nJ(1-\rho)} {\mathbb E}\left[\sum_{j=1}^J \sum_{i=1}^n [\rho (\varepsilon_{0,j}^{(2)})^2+ (1-\rho) (\varepsilon_{i,j}^{(2)})^2 ]\right] \leq  \frac 1{1-\rho} {\mathbb E}
[\varepsilon_{1,1}^2{\mathbf 1}_{|\varepsilon_{1,1}|>K}].$$
Therefore, by Cauchy-Schwarz and Assumption \ref{subExp}, we obtain
    \begin{align*}
        {\mathbb E}_2\leq\frac 1{1-\rho} \sqrt{{\mathbb E}(\varepsilon_{1,1}^4){\mathbb P}(|\varepsilon_{1,1}|>K)}\leq \frac{\sqrt{c_1{\mathbb E}(\varepsilon_{1,1}^4)}}{1-\rho}e^{-c_2K/2} .
    \end{align*}
    Consequently, to obtain the wished result, we need for this term that ${\mathbb E}[\varepsilon_{1,1}^4]<+\infty$ and 
    \begin{align*}
        {\mathbb P}(|\varepsilon_{1,1}|>K) \leq Ce^{-2\log(nJ)},
    \end{align*}
and we choose 
    \begin{equation}\label{ChoixdeK}
    K=\frac{4}{c_2} \log(nJ). 
    \end{equation}

Let us turn to ${\mathbb E}_1(m)$ for some $m\in\mathcal{M}_{n,J}^+$, keeping in mind that we will change $m$ into $m\vee m'$ in the end.
We apply Talagrand inequality. For this we first calculate the quantities ${\mathbb H}^2$, $v$ and $M_1$ upper bounds on the terms ${\mathbb T}_1, {\mathbb T}_2, {\mathbb T}_3$ defined hereafter.
The first term to bound is
\begin{eqnarray*}
 {\mathbb T}_1:=   {\mathbb E}_{{\mathbf X}}\left[\sup_{t\in S_m, \|t\|_{n,J,R}=1} \nu_{n,J,1}^2(t) \right] 
\end{eqnarray*}
We denote by ${\mathbf Z}_{\bullet,\bullet} := {\mathbf R}^{-1/2} {\mathbf M} \varepsilon_{\bullet,\bullet}={\mathbf R}^{-1/2}\uvec$, with decomposition ${\mathbf Z}_{\bullet,\bullet}={\mathbf Z}_{\bullet,\bullet}^{(1)} + {\mathbf Z}_{\bullet,\bullet}^{(2)}$ following the one of $\varepsilon_{\bullet,\bullet}$. Note that the entries  ${\mathbf Z}_{i,j}^{(1)}$ of ${\mathbf Z}^{(1)}_{\bullet,\bullet}$ are centered but not uncorrelated; only the components of ${\mathbf Z}_{\bullet,\bullet}$ are, as ${\mathbf Z}_{\bullet,\bullet}=  {\mathbf R}^{-1/2}\uvec$ are such that ${\rm Var}({\mathbf Z}_{\bullet,\bullet})=\sigma^2 {\rm Id}_{nJ\times nJ}$.

By Gram-Schmidt orthonormalization, we consider $\varphi_{k,R}$, an orthonormalization of the $\varphi_k$ with respect to the scalar product $\langle \,  . \, , \, . \, \rangle_{n,J,R}$, so that $t\in S_m$ with $\|t\|_{n,J,R}=1$ can be written as
$t=\sum_{k=1}^m a_{k,R}\varphi_{k,R}$ with $\|t\|_{n,J,R}^2=\sum_{k=1}^m a_{k,R}^2=1$ and $\|\varphi_{k,R}\|_{n,J,R}^2=(nJ)^{-1}\varphi_{k,R}(\Xvec)^\top {\mathbf R}^{-1} \varphi_{k,R}(\Xvec) = 1$, for any $k\in\mathbb{N}$.
We get using Cauchy-Schwarz inequality
$$ {\mathbb T}_1\leq {\mathbb E}_{{\mathbf X}}\left[\sup_{t=\sum_{k=1}^m a_{k,R}\varphi_{k,R}, \sum_{k=1}^m a^2_{k,R}=1} \left(\sum_{k=1}^m a_{n,R}\nu_{n,J,1}(\varphi_{k,R})\right)^2\right]\leq \sum_{k=1}^m {\mathbb E}_{{\mathbf X}}\left[ \nu_{n,J,1}^2(\varphi_{k,R})\right].$$
Now, relying on the independency of the $\varepsilon_{i,j}^{(1)}$, we compute
\begin{eqnarray*} {\mathbb E}_{{\mathbf X}}\left[ \nu_{n,J,1}^2(\varphi_{k,R})\right]&=& {\mathbb E}_{{\mathbf X}}\left[\left( \frac 1{nJ} \sum_{i=0}^n \sum_{j=1}^J \varepsilon_{i,j}^{(1)} \left[M^\top R^{-1}\varphi_{k,R}(X_{\bullet, j})\right]_{i+1} \right)^2\right]
\\ &=&  \frac 1{(nJ)^2} \sum_{i=0}^n \sum_{j=1}^J {\mathbb E}[(\varepsilon_{i,j}^{(1)})^2] 
\left[M^\top R^{-1}\varphi_{k,R}(X_{\bullet, j})\right]_{i+1}^2 \\
&\leq & \frac 1{(nJ)^2} \sum_{i=0}^n \sum_{j=1}^J {\mathbb E}[\varepsilon_{i,j}^2] 
\left[M^\top R^{-1}\varphi_{k,R}(X_{\bullet, j})\right]_{i+1} ^2
\end{eqnarray*}
using that $ {\mathbb E}[(\varepsilon_{i,j}^{(1)})^2] \leq  {\mathbb E}[\varepsilon_{i,j}^2]=\sigma^2$.
Therefore 
\begin{eqnarray*}  {\mathbb E}_{{\mathbf X}}\left[ \nu_{n,J,1}^2(\varphi_{k,R})\right] &\leq &
\frac 1{(nJ)^2} \sum_{i=0}^n \sum_{j=1}^J {\mathbb E}[\varepsilon_{i,j}^2] 
\left[M^\top R^{-1}\varphi_{k,R}(X_{\bullet, j})\right]_{i+1}^2 \\ &&= {\mathbb E}_{{\mathbf X}}\left[\left(\frac 1{nJ}\sum_{i=0}^n\sum_{j=1}^J \varepsilon_{i,j}\left[M^\top R^{-1}\varphi_{k,R}(X_{\bullet, j})\right]_{i+1}\right)^2\right].
\end{eqnarray*}
We obtain
$${\mathbb T}_1\leq \sum_{k=1}^m {\mathbb E}_{{\mathbf X}}\left[ \left( \frac 1{nJ} \uvec^\top \,{\mathbf R}^{-1}\varphi_{k,R}(\Xvec)  \right)^2\right].$$
We get, by using now the non correlation property
\begin{eqnarray*}
    {\mathbb T}_1 
    &\leq & \sum_{k=1}^m  {\mathbb E}_{{\mathbf X}}\left[\left( \frac 1{nJ} \sum_{j=1}^J \sum_{i=1}^n {\mathbf Z}_{i,j} \left[R^{-1/2} \varphi_{k,R}({\mathbf X}_{\bullet,j})\right]_i\right)^2\right]  \\
    &=& \sum_{k=1}^m  \frac 1{(nJ)^2} \sum_{j=1}^J \sum_{i=1}^n {\mathbb E}_{{\mathbf X}}({\mathbf Z}_{i,j}^2)  \left[R^{-1/2} \varphi_{k,R}({\mathbf X}_{\bullet, j})\right]_i^2 \\ &=&   \frac{\sigma^2}{(nJ)^2}  \sum_{k=1}^m  \sum_{j=1}^J \sum_{i=1}^n  \left[R^{-1/2} \varphi_{k,R}({\mathbf X}_{\bullet,j})\right]_i^2. 
\end{eqnarray*}
As 
$$ \sum_{j=1}^J \sum_{i=1}^n  \left[R^{-1/2} \varphi_{j,R}({\mathbf X}_{\bullet, j})\right]_i^2= \varphi_{j,R}({\mathbf X}_{\bullet,\bullet})^\top {\mathbf R}^{-1} \varphi_{j,R}({\mathbf X}_{\bullet,\bullet}) =nJ \|\varphi_{j,R}\|_{n,J,R}^2=nJ,$$
we obtain
$${\mathbb T}_1\leq \sigma^2 \frac m{nJ} :={\mathbb H}^2.$$
Next we consider the second term useful for applying Talagrand and use the same type of trick as previously. Further, denote by ${\rm Var}_{\mathbf X}$ the variance conditional on $\Xvec$. 
\begin{eqnarray*} 
{\mathbb T}_2&:=& \sup_{t\in S_m, \|t\|_{n,J,R}=1}\frac 1{nJ} \sum_{j=1}^J\sum_{i=0}^n {\rm Var}_{\mathbf X}\left(\varepsilon_{i,j}^{(1)} [M^\top R^{-1} t({\mathbf X}_{\bullet,j})]_{i+1}\right)\\
&=& \sup_{t\in S_m, \|t\|_{n,J,R}=1}\frac 1{nJ} {\rm Var}_{\mathbf X}\left(\sum_{j=1}^J \sum_{i=0}^n \varepsilon_{i,j}^{(1)} \left[M^\top R^{-1} t({\mathbf X}_{\bullet,j})\right]_{i+1}\right) \\ &\leq & \sup_{t\in S_m, \|t\|_{n,J, R}=1}\frac 1{nJ} {\rm Var}_{\mathbf X}\left( \varepsilon_{\bullet,\bullet}^\top M^\top R^{-1} t({\mathbf X}_{\bullet,\bullet})\right)  =
\sup_{t\in S_m, \|t\|_{n,J,R}=1}\frac 1{nJ} {\rm Var}_{\mathbf X}\left({\mathbf Z}_{\bullet,\bullet}^\top  {\mathbf R}^{-1/2} t({\mathbf X}_{\bullet,\bullet})\right)  
\\ &=& \sup_{t\in S_m, \|t\|_{n,J,R}=1}\frac 1{nJ} {\rm Var}_{\mathbf X}\left(\sum_{j=1}^J \sum_{i=1}^n {\mathbf Z}_{i,j}  [R^{-1/2} t({\mathbf X}_{\bullet,j})]_i\right) \\ & = & \sigma^2 \sup_{t\in S_m, \|t\|_{n,J,R}=1}\frac{1}{nJ} \sum_{j=1}^J \sum_{i=1}^n   [R^{-1/2} t({\mathbf X}_{\bullet,j})]_i^2 = \sigma^2 :=v, 
\end{eqnarray*}
where we used that $\sum_{j=1}^J \sum_{i=1}^n   [R^{-1/2} t({\mathbf X}_{\bullet,j})]_i^2 =nJ \|t\|_{n,J,R}^2$.
 The last term is computed as follows. 
\begin{eqnarray*}
   {\mathbb T}_3&:=& \sup_{t\in S_m, \|t\|_{n,J,R}=1}\sup_{e\in {\mathbb R}} \; \sup_{j=1, \dots, J}  \sup_{i=0, \dots, n} |e{\mathbf 1}_{|e|\leq K} - {\mathbb E}(\varepsilon_{1,1}{\mathbf 1}_{|\varepsilon_{1,1}|\leq K} )| [M^\top R^{-1} t({\mathbf X}_{\bullet,j})]_{i+1}| \\ &\leq & 2K
   \sup_{t\in S_m, \|t\|_{n,J,R}=1} \; \sup_{i,j} | [M^\top R^{-1} t({\mathbf X}_{\bullet,j})]_{i+1}| 
\end{eqnarray*}
Now we write that
\begin{eqnarray*}
    | [M^\top R^{-1} t({\mathbf X}_{\bullet,j})]_{i+1}|  & \leq & 
    \sqrt{ \sum_{j=1}^J\sum_{i=0}^n [M^\top R^{-1} t({\mathbf X}_{\bullet,j})]_{i+1}^2} =\sqrt{{\rm Var}_{\mathbf X}\left( \sum_{j=1}^J\sum_{i=0}^n \varepsilon_{i,j} [M^\top R^{-1} t(X_{\bullet, j})]_{i+1}\right)} \\ &=& \sqrt{{\rm Var}_{\mathbf X}( {\mathbf Z}_{\bullet,\bullet}^\top {\mathbf R}^{-1/2} t({\mathbf X}_{\bullet,\bullet}))}= \sqrt{t(\Xvec)^\top {\mathbf R}^{-1/2} {\rm Var}({\mathbf Z}_{\bullet,\bullet}) {\mathbf R}^{-1/2}  t(\Xvec)}  \\ &=&  = \sqrt{nJ}  \|t\|_{n,J,R} = \sqrt{nJ}
\end{eqnarray*}
for $t$ such that $\|t\|_{n,J,R}=1$. Finally, we obtain the bound ${\mathbb T}_3\leq 2K\sqrt{nJ}:=M_1$.
    
Now we obtain with Talagrand inequality for any $m\in{\mathcal M}_n^+$
$${\mathbb E}_1(m) \leq \frac{C_1}{nJ} \left( e^{-C_2m} + 4K^2 e^{-C_3\sqrt{m}/(2K)}\right) .$$
Therefore, it follows for any $m\in\mathcal{M}_{n,J}$
\begin{eqnarray*} \sum_{m'\in {\mathcal M}_n^+} {\mathbb E}_1(m\vee m')  &\leq & \sum_{m'\in {\mathcal M}_n^+} {\mathbb E}_1(m')  
\leq  \frac{C_1}{nJ} \left( \sum_{m\geq 0} e^{-C_2m} + 4K^2\sum_{m\in {\mathcal M}_n^+}  e^{-C_3\sqrt{m}/(2K)}\right) \\
&\leq &  \frac{C_1}{nJ} \left(\Sigma(C_2) + 4K^2\sum_{\ell=1}^{c (nJ)^2} e^{-C_3\ell /(2K)}\right)\\
&\leq &  \frac{C_1}{nJ} \left(\Sigma(C_2) + \frac{4K^2}{1-e^{-C_3/(2K)}}\right)
\leq  \frac{C_1}{nJ} \left(\Sigma(C_2) + \frac{8}{C_3} K^3\right),
\end{eqnarray*} 
as for $x>0$, it holds that $1/(1-e^{-x})\leq 1/x$. Now, plugging the choice of $K$ given in  (\ref{ChoixdeK}) yields 
$$\sum_{m'\in {\mathcal M}_n^+} {\mathbb E}_1(m\vee m')\leq C\frac{\log^{3}(nJ)}{nJ},$$
for some constant $C$ depending on $C_1, C_2, C_3$, the constants of the Talagrand and $ c_1$ and $c_2$ the constants of Assumption \ref{subExp}.
\end{proof}

\subsection{Proof of (\ref{MajProb})}\label{sec::equ-proof}
\begin{proof}[Proof of (\ref{MajProb})]
We write that
$${\mathbb P}(|\varepsilon_{1,1}|> K)={\mathbb E}({\mathbf 1}_{|\varepsilon_{1,1}|>K})= {\mathbb E}({\mathbf 1}_{\varepsilon_{1,1}>K}{\mathbf 1}_{\varepsilon_{1,1}>0}) + {\mathbb E}({\mathbf 1}_{\varepsilon_{1,1}<-K}{\mathbf 1}_{\varepsilon_{1,1}\leq 0}) : ={\mathbb M}_1+ {\mathbb M}_2.$$
Then, for all $u\geq 0$,
$${\mathbb M}_1 \leq {\mathbb E}(e^{u\varepsilon_{1,1}-uK} {\mathbf 1}_{\varepsilon_{1,1}>0}) \leq e^{-uK+\frac{u^2s^2}2}$$
by using (\ref{HypSubGauss}). By optimization with respect to $u$, we choose $u=K/s^2$ and obtain ${\mathbb M}_1\leq e^{-K^2/(2s^2)}$.
Similarly, for $u<0$, ${\mathbb M}_2\leq e^{uK+u^2s^2/2}$ and choosing $u=-K/s^2$ yields ${\mathbb M}_2\leq e^{-K^2/(2s^2)}$, and the result follows. 
\end{proof}

\section*{Appendix: auxiliary results}
The following result of \cite{tropp2012user} is used.
\begin{lem}[Tropp Chernov]\label{lem::tropp-chernoff}
    Consider a fixed sequence $(\bm X_k)$ of independent, random, self-adjoint matrices with dimension $d$. Let $\lambda_{\max}(\bm X_k)$ and $\lambda_{\min}(\bm X_k)$ denote the largest and smallest eigenvalue of $\bm X$, respectively. Further, let $\preccurlyeq$ denote the semi-definite order on self-adjoint matrices.  Assume that each random matrix satisfies
    \begin{align*}
        \bm X_k \preccurlyeq 0 \quad\text{ and }\quad \lambda_{\max}(\bm X_k) \leq R, \quad\text{ almost surely.}
    \end{align*}
    Define 
    \begin{align*}
        \mu_{\min} := \lambda_{\min}\left(\sum_{k\in\mathbb{N}} \E [\bm X_k]\right) \quad\text{ and }\quad  \mu_{\max} := \lambda_{\max}\left(\sum_{k\in\mathbb{N}} \E [\bm X_k]\right).
    \end{align*}
    Then, 
    \begin{align*}
        \mathbb{P} \left(\lambda_{\min}\left(\sum_{k\in\mathbb{N}} \bm X_k\right) \leq (1-\delta)\mu_{\min}\right)&\leq d \cdot \left(\frac{e^{-\delta}}{(1-\delta)^{1-\delta}}\right)^{\mu_{\min}/R}, \quad\text{ for }\delta\in[0,1], \text{ and}\\
        \mathbb{P} \left(\lambda_{\max}\left(\sum_{k\in\mathbb{N}} \bm X_k\right) \leq (1 + \delta)\mu_{\max}\right)&\leq d \cdot \left(\frac{e^{\delta}}{(1-\delta)^{1+\delta}}\right)^{\mu_{\max}/R}, \quad\text{ for }\delta\geq 0.
    \end{align*}
\end{lem}

For the proof of Proposition \ref{prop::new-prop3} we use the following helpful result, which can be found e.g. in \cite{Bartlett1951}.

\begin{lem}[Sherman-Morrison formula]\label{lem::sherman-morrison}
    Let $A\in\Rz^{n\times n}$ be an invertible square matrix and $u,v\in\Rz^n$. Then $A + uv^\top$ is invertible if and only if $1 + v^\top A^{-1} u \not= 0$. In this case, it holds
    \begin{align*}
        \left(A + uv^\top\right)^{-1} = A^{-1} - \frac{A^{-1} u v^\top A^{-1}}{1 + v^\top A^{-1} u}.
    \end{align*}
\end{lem}

\bibliography{bib-common-noise}
\end{document}